\documentclass[a4paper,11pt]{amsart}

\usepackage[top=1.5in, bottom=1in, left=0.9in, right=0.9in]{geometry}

\usepackage{amsfonts,latexsym,rawfonts,amsmath,amssymb,amsthm, mathrsfs, lscape}
\usepackage[all]{xy}
\usepackage[english]{babel}
\usepackage[utf8]{inputenc}
\usepackage{xfrac}

\usepackage{array, tabularx, mathtools}

\usepackage{textcomp}

\usepackage[backref=page]{hyperref}
\renewcommand*{\backref}[1]{}

\usepackage[svgnames]{xcolor}    
\hypersetup{
     colorlinks   = true,
     citecolor    = Green
}
\usepackage{relsize}

\newtheorem{thm}{Theorem}[section]
\newtheorem*{thm*}{Theorem}
\newtheorem{cor}[thm]{Corollary}
\newtheorem*{cor*}{Corollary}
\newtheorem{lem}[thm]{Lemma}
\newtheorem*{lem*}{Lemma}

\newtheorem{prop}[thm]{Proposition}
\newtheorem*{prop*}{Proposition}

\newtheorem*{conj*}{Conjecture} 
\newtheorem{rmk}[thm]{Remark}

\newtheorem{defi}[thm]{Definition}

\newcommand{\CC}{\mathbb{C}}
\newcommand{\HH}{\mathbb{H}}
\newcommand{\RR}{\mathbb{R}}
\newcommand{\Ss}{\mathbb{S}}
\newcommand{\DD}{\mathbb{B}}
\newcommand{\D}{{D}}
\newcommand{\BB}{\mathbb{B}}
\newcommand{\ZZ}{\mathbb{Z}}
\newcommand{\NN}{\mathbb{N}}
\newcommand{\PP}{\mathbb{P}}

\newcommand{\QQ}{\mathbb{Q}}

\newcommand{\Cn}[1]{(\mathbb{C}^*)^{#1}}
\newcommand{\Aut}{\mathrm{Aut}}
\newcommand{\GL}{\mathrm{GL}}
\newcommand{\e}{\mathrm{e}}
\newcommand{\id}{\operatorname{\text{\sf id}}}
\newcommand{\del}{\partial}
\newcommand{\im}{\mathop\mathrm{im}}
\newcommand{\rank}{\mathop\mathrm{rank}}

\newcommand{\al}{\alpha}
\newcommand{\be}{\beta}
\newcommand{\la}{\lambda}

\newcommand{\re}{\mathop\mathrm{Re}}
\newcommand{\I}{\mathbb{I}}
\newcommand{\CP}{\CC\PP}
\newcommand{\Spec}{\mathop\mathrm{Spec}}
\newcommand{\Ka}{K\"ahler}
\newcommand{\ce}{\mathcal{C}^\infty}

\newcommand{\overbar}[1]{\mkern 1.5mu\overline{\mkern-1.5mu#1\mkern-1.5mu}\mkern 1.5mu}

\def\XXint#1#2#3{{\setbox0=\hbox{$#1{#2#3}{\int}$ }
\vcenter{\hbox{$#2#3$ }}\kern-.6\wd0}}

\allowdisplaybreaks[1]

\title{On a class of Kato manifolds}

\author{Nicolina Istrati}
\address[Nicolina Istrati]{School of Mathematical Sciences, Tel Aviv University, Ramat Aviv, Tel Aviv 69978 Israel}
\email{nicolinai@mail.tau.ac.il}

\author{Alexandra Otiman}
\address[Alexandra-Iulia Otiman]{Roma Tre University, Department of Mathematics and Physics, Largo San
Leonardo Murialdo, Rome, Italy AND
Institute of Mathematics “Simion Stoilow” of the Romanian Academy, 21, Calea Grivitei,
010702, Bucharest, Romania AND
University of Bucharest, Research Center in Geometry, Topology and Algebra, Faculty of Mathematics and Computer Science, 14 Academiei Str., Bucharest, Romania
}
\email{aiotiman@mat.uniroma3.it}

\author{Massimiliano Pontecorvo}
\address[Massimiliano Pontecorvo]{Roma Tre University, Department of Mathematics and Physics, Largo San
Leonardo Murialdo, Rome, Italy}
\email{max@mat.uniroma3.it}

\thanks{A. O. is partially supported by a grant of the Romanian Ministry of Research and Innovation, CNCS - UEFISCDI,
project number PN-III-P4-ID-PCE-2016-0065, within PNCDI III}

\begin{document}

\begin{abstract}
We revisit Brunella's proof of the fact that Kato surfaces admit locally conformally K\" ahler metrics, and we show that it holds for a large class of higher dimensional complex manifolds containing a global spherical shell. On the other hand, we construct manifolds containing a global spherical shell which admit no locally conformally K\"ahler metric. We consider a specific class of these manifolds, which can be seen as a higher dimensional analogue of Inoue-Hirzebruch surfaces, and study several of their analytical properties. In particular, we give new examples, in any complex dimension $n \geq 3$, of compact non-exact locally conformally K\" ahler manifolds with algebraic dimension $n-2$, algebraic reduction bimeromorphic to $\CC\PP^{n-2}$ and admitting non-trivial holomorhic vector fields.
\end{abstract}

\maketitle

\section*{Introduction}

\bigskip
In \cite{k} Masahide Kato introduced a class of compact complex manifolds of non-K\" ahler type, which can be described as containing a {\it global spherical shell}.  They are often referred to in the literature as {\it GSS manifolds} or {\it Kato manifolds}, especially in the intensively studied case of complex dimension 2. We shall follow the exposition in \cite{k} to introduce them.

A spherical shell (SS) in an $n$-dimensional complex manifold is an open subset biholomorphic to a neighbourhood of the sphere $\Ss^{2n-1}$ in $\CC^{n}$. A global spherical shell (GSS) is a spherical shell such that its complement is connected. The simplest examples of manifolds containing a GSS are the primary Hopf manifolds and their blow-ups. More generally, Kato showed in \cite{k} that any compact complex manifold containing a GSS arises in the following way. Let $\pi:\hat\DD\rightarrow \DD$ be a modification of the unit ball in $\CC^n$ at finitely many points and let $\sigma:\overline\DD\rightarrow \hat\DD$ be a holomorphic embedding. Then glue the two boundary components of $\hat\DD-\sigma(\overline\DD)$ via the real analytic CR-diffeomorphism $\sigma \circ \pi$. The manifold $M$ such obtained is a smooth $n$-dimensional compact complex manifold with $\pi_1(M)\cong\ZZ$ and contains a GSS given by a neighbourhood of $\del\hat\DD$. The couple $(\pi,\sigma)$ is called a Kato data for the manifold $M$.

Since Kato manifolds have first Betti number  equal to $1$, they cannot have a K\" ahler metric. Little is known in general about the Hermitian geometry of complex non-K\" ahler manifolds. The special Hermitian non-K\" ahler metrics studied so far usually arise by imposing specific cohomological conditions on the fundamental $(1,1)$-form of the metric. One such class is given by the {\it locally conformally K\" ahler (lcK)} metrics (see Section~\ref{SecLCK}), which are of  particular interest on compact non-K\"ahler manifods. Indeed, this is due to a theorem of Vaisman \cite{vai80}, stating that on a compact K\" ahlerian manifold, any lcK metric is automatically globally conformal to a K\"ahler metric. 

In complex dimension $2$, lcK metrics were first constructed on some particular classes of Kato surfaces. LeBrun \cite{lb} gave a construction for certain \textit{parabolic Inoue} surfaces and Fujiki-Pontecorvo \cite{fp10}, on all Inoue-Hirzebruch surfaces, using a twistor construction.  Shortly after, Brunella in \cite{bru10} and in \cite{bru} showed that all Kato surfaces admit lcK metrics. In the present paper, we extend Brunella's second construction of lcK metrics to a rather large class of Kato manifolds in any complex dimension. In particular, they give new examples of lcK manifolds which do not admit lcK metrics with potential. So far, the only other known such examples in high dimension were the Oeljeklaus-Toma manifolds of type $(s,1)$ (\cite{ot05}, \cite{o16}) and the blow-ups of lcK manifolds (\cite{tri82}, \cite{vuli}, \cite{ovv}).

\begin{thm*}[see Theorem~\ref{LCK}]
Let $M$ be a Kato manifold of dimension $n$ with Kato data $(\pi,\sigma)$, so that $\pi$ is a finite sequence of blow-ups along smooth centers. Then $M$ admits a locally conformally K\"ahler metric. 
\end{thm*}

Note that for $n=2$, the hypothesis in the above for $\pi$ is not needed since any modification is a sequence of blow-ups at points. However, in higher dimension, Hironaka's famous examples of modifications \cite{hirn} allow us to construct examples showing that the hypothesis on $\pi$ is necessary:

\begin{prop*}[see Proposition~\ref{NOLCK}]
In any dimension $n\geq 4$ there exist Kato manifolds which admit no lcK metric.
\end{prop*}

By now, we have a very good understanding of the analytical properties of Kato surfaces, by the works of Nakamura (\cite{n84}, \cite{n90}), Dloussky (\cite{dl}, \cite{d1}, \cite{d2}), Dloussky-Oeljeklaus (\cite{do99}) etc. The higher dimensional case is less understood due to the numerous ways one can perform modifications. An explicit example of a Kato threefold together with the computation of some of its analytical and topological invariants is given by Ruggiero in \cite{r}. The main feature which differentiates Kato surfaces from their higher dimensional analogues is the fact that, while the germ $\pi \circ \sigma: (\CC^n, 0) \rightarrow (\CC^n, 0)$ uniquely determines the surface, this no longer holds for $n>2$. An example which illustrates this fact for $n=3$ appears in \cite{r}. 

In the present paper, we consider the simplest class of Kato manifolds, namely when $\pi$ in the Kato data is given by a sequence of blow-ups at points which are centers of the standard charts of $\mathrm{Bl}_0\DD$ and $\sigma$ is itself given by a standard chart. The Kato data in this case can be encoded by a matrix $A\in\GL(n,\ZZ)$, which we call a \textit{Kato matrix}, and the resulting Kato manifold $M_A$ is completely determined by $A$.  In particular, the analytical properties of the manifolds $M_A$ can be read off the algebraic properties of $A$. The germ $F=\pi\circ\sigma$ corresponding to such a Kato data has only monomial components and has a simple expression in terms of $A$. In addition, in this case the germ uniquely determines the Kato manifold in our specific class. 

This class of manifolds intersects the one constructed by Tsuchihashi in \cite{tsu} as toroidal compactifications of $\ZZ$-quotients of open subsets in complex tori. In complex dimension $2$, the manifolds $M_A$ are Inoue-Hirzebruch surfaces, by \cite{d1}.  In a description coming from number theory, due to Hirzebruch (see \cite{h}), these surfaces appear as compactifications of  quotients of $\mathbb{H} \times \mathbb{C}$ by $\Lambda \rtimes U$, where $\Lambda$ is a finite index lattice in the ring of integers of some quadratic field $K$ and $U$ is a cyclic group of positive units in $K$. This generalizes to any dimension in the following manner (see Theorem~\ref{Thmcompact} for a more precise statement):

\begin{thm*}
The Kato manifold $M_A$ is a compactification of $\CC^{n-1}\times\HH/\Lambda\rtimes U$ with rational hypersurfaces, where $U\cong\ZZ$, $\Lambda\subset\CC^n$ is a lattice of rank $r\in\{2\ldots, n\}$, and the group $\Lambda\rtimes U$ is determined by $A$. 
\end{thm*}

The Kodaira dimension of the manifolds we consider is always $-\infty$ (Proposition~\ref{Kodairadimension}) and they do not admit holomorphic one-forms (Proposition~\ref{forme}). What is remarkable about this class of Kato manifolds is that it gives examples, in any complex dimension $n\geq 3$, of lcK manifolds without potential (Proposition~\ref{nopotential}) which admit non-trivial holomorphic vector fields (Proposition~\ref{vf}) and have positive algebraic dimension (Proposition~\ref{algdim}). 

Call $r$ in the above thorem the rank of the manifold $M_A$. When $r=2$, we find the following algebraic description of $M_A$ (see Proposition~\ref{algdim} and Theorem~\ref{ThmAlgRed}):

\begin{thm*}
Let $M_A$ be $n$-dimensional, of rank $2$. Then the algebraic dimension of $M_A$ is $n-2$. Its algebraic reduction is bimeromorphic to $\CC\PP^{n-2}$, with a generic fiber being bimeromorphic to a Kato surface $M_B$.
\end{thm*}

Concerning the holomorphic vector fields, we find (cf. Corollary~\ref{torus} and Theorem~\ref{ThmHvn-2} ):

\begin{thm*}
Let $M_A$ be $n$-dimensional of rank $r$. Then the compact torus $\mathbb{T}^{n-r}$ acts effectively by biholomorphisms on $M_A$. Moreover, if $r=2$, then $\dim H^0(M_A,TM_A)=(n-1)(n-2)$. 
\end{thm*}

The paper is organized as follows. In Section~\ref{SecKato} we recall the construction of a general Kato manifold and discuss some of its features. In Section~\ref{SecLCK}, after some preliminaries on locally conformally K\" ahler geometry, we present Brunella's proof in any complex dimension, as well as an example of a Kato manifold without lcK metric. We briefly describe in Section~\ref{SecCoh} the computations for the Betti numbers and the twisted Betti numbers of some Kato manifolds. In Section~\ref{SecIH}, we define a specific class of matrices in $\GL(n,\ZZ)$, the Kato matrices. We study some of their algebraic properties which we need later.  Then, we describe the construction through which we uniquely associate to any Kato matrix $A$ a Kato manifold $M_A$.  In Section~\ref{secComp} we prove a compactification result for these manifolds. The next two sections are dedicated to understanding the analytical properties of the manifolds $M_A$, such as the Kodaira dimension, the algebraic dimension and the space of holomorphic vector fields and one-forms. Finally, in Section~\ref{secn-2} we focus on the class of rank 2 manifolds, for which we find the description of the algebraic reduction as well as the space of holomorphic vector fields. 

\subsection*{Notation} 
Consider the following two norms on $\CC^n$. For $z=(z_1,\ldots,z_n)\in\CC^n$, we put $||z||^2:=\sum_{j=1}^n|z_j|^2$ and $||z||^2_{1,2}:=\sum_{j=1}^{n-1}|z_j|^2+2|z_n|^2$. For $c>0$, we denote by $\DD_c:=\{z\in\CC^n, ||z||<c\}$ and by $\DD_{1,2,c}:=\{z\in\CC^n, ||z||_{1,2}<c\}$ the ball of radius $c$ for the first norm and the second one respectively. We will also write $\DD:=\DD_1$ and $\DD_{1,2}:=\DD_{1,2,1}$. Moreover, we denote by $S_c:=\{z\in\CC^n, 1-c<||z||<1+c\}$. For a subset $S$ of a topological space, we denote by $\mathring S$ its interior and by $\overline{S}$ its closure. For a map $F:\CC^n\rightarrow \CC^n$ and $m>0$, we will denote by $F^m:=F\circ\ldots \circ F$, where the composition is taken $m$ times.

\section{Kato manifolds}\label{SecKato}

In this section, we recall the construction of a general Kato manifold in more detail, following \cite{k} and \cite{dl}. 

\begin{defi} A {\it spherical shell} in a compact complex manifold $M$ of dimension $n$ is an open subset $V \subset M$ such that $V$ is biholomorphic to $S_\epsilon$ for some $\epsilon >0$. We say that $V$ is a {\it global spherical shell (GSS)} if $M \setminus V$ is connected. A manifold which contains a GSS is called a {\it Kato manifold}. 
\end{defi}

\begin{defi} Let $M$ be a complex manifold and let $p_1, \ldots, p_m$ be points in $M$. A {\it proper modification} of $M$ at $p_1, \ldots, p_m$ is a proper surjective holomorphic map $\pi: \hat{M} \rightarrow M$ such that $E:=\pi^{-1}(p_1) \cup \ldots \cup \pi^{-1}(p_m)\subset \hat M$ is a nowhere dense analytic subset and $\pi$ is a biholomorphism precisely outside $E$. $E$ is called the \textit{exceptional set} of the modification. 
\end{defi}

As we have seen in the introduction, any Kato manifold is constructed starting from a modification $\pi:\hat\BB\rightarrow\BB\subset\CC^n$ at finitely many points and an embedding $\sigma:\overline\BB\rightarrow \hat\BB$. We shall call $(\pi, \sigma)$ a {\it Kato data} and $F:=\pi\circ\sigma:\BB\rightarrow\BB$ the {\it corresponding germ}. Let: 
\begin{gather*}
W:=\hat\BB-\sigma(\BB),\ \ \  \del_+W:=\del\hat\BB,\ \ \ \del_-W:=\del\sigma(\DD)\\
\gamma:\del_+ W\rightarrow \del_-W, \ \ \gamma:=\pi\circ\sigma|_{\del_+W}.
\end{gather*}
Clearly $\gamma$ extends to a biholomorphism between small neighbourhoods of the two components of $\del W$, and we obtain the Kato manifold:
\begin{equation*}
M(\pi,\sigma):=\overline W/_\sim
\end{equation*}
where $ x\sim y$  if $ x\in\del_+W$ and $y=\gamma(x)\in\del_-W$.

In order to describe the universal cover of $M=M(\pi,\sigma)$, for each $m\in\ZZ$, let $W_m$ be a copy of $W$. Then the universal cover can be seen as:
\begin{equation}\label{univc}
\tilde {M}= \left(\bigsqcup_{m \in \mathbb{Z}} \overline W_m \right) \big/_\sim =\bigsqcup_{m\in\ZZ} W_m
\end{equation}
where $\sim$ is given by identifying $x\in\del_+W_m$ with $\sigma\circ\pi(x)\in \del_-W_{m+1}$. In particular, we have a natural isomorphism $\pi_1(M)\cong \ZZ$ with respect to which $1$ acts on $\tilde M$ by sending an element in $W_m$ to its copy in $W_{m+1}$, so that $W\subset \tilde M$ is a fundamental domain for this action. Denote by $q:\tilde M\rightarrow M$ the covering map.

Let $(\pi,\sigma)$ be a Kato data, let $P\subset \BB$ denote the finite set of points modified by $\pi$ and let $E\subset\hat\BB$ denote the exceptional set of $\pi$. If $\sigma(P)$ does not intersect $E$, then the Kato manifold $M(\pi,\sigma)$ is a modification of a primary Hopf manifold \cite[Proposition~1]{k}. Thus, in all that follows, we will always assume that $\sigma(P)\cap E\neq \emptyset$ for a Kato data.

A Kato data $(\pi,\sigma)$ is called \textit{centered} if $\pi\circ\sigma(0)=0$. By the following result, whose proof is independent of the dimension, one can always suppose that a Kato data is centered:

\begin{lem}(\cite[Lemme~1.5, Part~I]{dl})
Given a Kato data $(\pi,\sigma)$, there exists a centered Kato data $(\pi',\sigma')$ and a biholomorphism $M(\pi,\sigma)\cong M(\pi',\sigma')$.
\end{lem}

We will call a Kato data $(\pi,\sigma)$ \textit{simple} if it is centered and $\pi$ is a biholomorphism outside $0$. The following result, due to Dloussky in the case of surfaces, readily adapts to higher dimension. It shows that it is not a big loss of generality to consider only simple Kato data.

\begin{lem}(\cite[Lemme~2.7, Part~I]{dl})
Let $(\pi,\sigma)$ be a centered Kato data. Then there exists a simple Kato data $(\pi',\sigma')$ and a proper modification at finitely many points $\mu:M(\pi, \sigma)\rightarrow M(\pi',\sigma')$.
\end{lem}

\begin{rmk}\label{Schwarz}
Let $(\pi,\sigma)$ be a simple Kato data and let $F:\BB\rightarrow \BB$ be the corresponding germ, with $F(0)=0$. Let $D\subset \BB$ be a domain with smooth boundary which contains $0$ and satisfies $F(\overline{D})\subset D$, or equivalently, $\sigma(\overline D)\subset \pi^{-1}(D)$. Let $\hat D:=\pi^{-1}(D)$ and $W_D:=\hat D-\sigma(D)$. Then the manifold $M_D:=W_D/_\sim$, defined by $x\sim y$ if $y=\sigma\circ\pi(x)\in\del W_D$, is biholomorphic to $M(\pi,\sigma)$. Indeed, clearly the universal covers of the two manifolds coincide, and $W_D$ is nothing but another fundamental domain for the action of $\pi_1(M(\pi,\sigma))=\pi_1(M_D)$, seen as a deck group. In particular,  by the Schwarz lemma (see for instance \cite[Theorem~6]{shabat}), for any $r<1$ one has $F(\overline\BB_r)\subset\BB_r$. Thus, considering $\hat\BB_r=\pi^{-1}(\BB_r)$ and the new Kato data $(\pi_r:=\pi|_{\hat\BB_r},\sigma_r:=\sigma|_{\BB_r})$, the manifold $M(\pi_r,\sigma_r)$ is biholomorphic to $M(\pi,\sigma)$. This remark is at the heart of Brunella's construction of lcK metrics on Kato manifolds which we present in the next section.
\end{rmk}

In practice it is convenient to allow for more general domains of definition for a Kato data than just $\BB$ and its modifications, as long as the existence of a GSS is guaranteed. Suppose we are given a relatively compact domain with smooth boundary $D\subset \CC^n$ which contains $0$. Let $\pi:\hat D\rightarrow D$ be a proper modification at $0$ and let $\sigma:\overline D\rightarrow D$ be a holomorphic embedding, sending $0$ to the exceptional set of $\pi$. Then one can define $W_D:=\hat D-\sigma(D)$ and the compact complex manifold $M(\pi,\sigma):=W_D/_\sim$ as before. Let $F=\pi\circ\sigma$ be the corresponding germ. If there exists an open set $B\subseteq D$ which contains $0$, is biholomorphic to $\BB$ and satisfies $F(\overline B)\subset B$, then a neighbourhood of $\pi^{-1}(\del B)\subset \overline{W}_D-\sigma(\del D)\subset M(\pi,\sigma)$ is a GSS, hence $M(\pi,\sigma)$ is a Kato manifold. In this case, we will call $(\pi:\hat D\rightarrow D,\sigma:\overline D\rightarrow \hat D)$ a Kato data as well.

\subsection*{Composing Kato data}  

Given two Kato data $(\pi_{j}:\hat\D_{j}\rightarrow\D\subset\CC^n,\sigma_j:\overline{D}\rightarrow \hat D_j)$ with corresponding germs $F_j=\pi_j\circ\sigma_j$, $j=1,2$, one can glue the two together in order to form a new Kato data $(\pi_{12},\sigma_{12})$ with corresponding germ $F_1\circ F_2$. We will call the resulting data the composition of $(\pi_1,\sigma_1)$ with $(\pi_2,\sigma_2)$. This is done as follows. Let $W_j:=\hat\D_j-\sigma_j(\D)$ for $j=1,2$. The map $\gamma_1:=\sigma_1\circ\pi_2:\del\hat\D_2\rightarrow\del\sigma_1(\D)$ extends to a biholomorphism of small enough neighborhoods of $\del\hat\D_2$ and $\del\sigma_1(\D)$, so we can define:

\begin{align}\label{glue}
\begin{split}
\hat\D_{12}&:=W_1\sqcup_{\gamma_1}\hat\D_2\\
W_{12}&:=W_1\sqcup_{\gamma_1}W_2.
\end{split}
\end{align}
Let $\iota:\hat\D_2\rightarrow \hat\D_{12}$, $j:W_1\rightarrow \hat\D_{12}$ denote the natural embeddings and let $\sigma_{12}:=\iota\circ\sigma_2:\overline{\D}\rightarrow \hat\D_{12}$, so that $W_{12}=\hat\D_{12}-\sigma_{12}(\D)$. Moreover, define $\pi'_2:\hat\D_{12}\rightarrow \hat\D_1$ by:
\begin{equation*}
\pi'_2|_{j(W_1)}:=\id, \ \ \pi'_2|_{\iota(\hat\D_2)}:=\sigma_1\circ\pi_2\circ\iota^{-1}.
\end{equation*}
Then $\pi'_2$ is precisely the modification $\pi_2$ of the chart $\sigma_1(\D)\subset\hat\D_1$. It follows that $\pi_{12}:=\pi_1\circ\pi'_2:\hat\D_{12}\rightarrow\D$ is a proper modification of $\D$ at finitely many points, hence $(\pi_{12},\sigma_{12})$ is a Kato data. Its corresponding germ is given by:
\begin{equation*}
F_{12}:=\pi_{12}\circ\sigma_{12}=\pi_1\circ\pi'_2\circ\iota\circ\sigma_2=\pi_1\circ\sigma_1\circ\pi_2\circ\sigma_2=F_1\circ F_2.
\end{equation*}

While clearly the resulting Kato data depends on the order in which we performed the gluing, the corresponding manifold does not. More precisely, let us denote by $(\pi_{21},\sigma_{21})$ the composition of $(\pi_2,\sigma_2)$ with $(\pi_1,\sigma_1)$. Then we have:
\begin{lem}\label{ciclic}
There exists a biholomorphism $M(\pi_{12},\sigma_{12})\cong M(\pi_{21},\sigma_{21})$.
\end{lem}
\begin{proof}
Let us denote by $M_{12}=M(\pi_{12},\sigma_{12})$ and by $M_{21}=M(\pi_{21},\sigma_{21})$. By \eqref{univc}, we have:
\begin{equation*}
\tilde M_{12}=\bigsqcup_{m\in\ZZ} W_{m,12}=\bigsqcup_{m\in\ZZ}(W_{m,1}\sqcup W_{m,2})=\bigsqcup_{m\in\ZZ}(W_{m,2}\sqcup W_{m,1})=\tilde M_{21}
\end{equation*}
where for each $m\in\ZZ$ and $j\in\{1,2,12\}$, $W_{m,j}$ denotes a copy of $W_j$. In particular, $W_{12}$ and $W_{21}$ are two different fundamental domains for the action of the deck group $\pi_1(M_{12})=\pi_1(M_{21})$, thus $M_{12}\cong M_{21}$.
\end{proof}

\begin{lem}\label{finitecov}
Let $(\pi,\sigma)$ be a Kato data, let $p\geq 2$ and let $(\pi_p:\hat\D_p\rightarrow\D,\sigma_p)$ be the Kato data obtained by composing $(\pi,\sigma)$ with itself $p$ times, as described above. Then the manifold $M_p:=M(\pi_p,\sigma_p)$ is a finite cyclic unramified covering of $M:=M(\pi,\sigma)$ with $p$ sheets.
\end{lem}
\begin{proof}
Let $W:=\hat\D-\sigma(\D)$ and let $\gamma:=\sigma\circ\pi:\del_+W\rightarrow \del_-W$. Applying inductively \eqref{glue}, we find that $W^{(p)}:=\hat\D_{p}-\sigma_p(\D)=W\sqcup_{\gamma}\ldots\sqcup_{\gamma} W$, where the copies of $W$ are taken $p$ times. Moreover, we have 
\begin{equation*}
\gamma_p:=\sigma_p\circ\pi_p=\gamma^p:\del_+W_p\rightarrow\del_-W_p.
\end{equation*}
Thus, from \eqref{univc} we find:
\begin{equation*}
\tilde M_p=\bigsqcup_{m\in\ZZ}W^{(p)}_{m}=\bigsqcup_{m\in\ZZ}(W_m\sqcup\ldots \sqcup W_m)=\tilde M.
\end{equation*}
In addition, under the natural isomorphism $\pi_1(M)\cong\ZZ$, we find that $p\ZZ\cong\pi_1(M_p)\subset\pi_1(M)$, therefore we have a Galois covering $M_p\rightarrow M$ of deck group $\pi_1(M)/\pi_1(M_p)\cong\ZZ/p\ZZ$.
\end{proof}

\section{Locally conformally K\" ahler metrics on Kato manifolds: Brunella's proof revisited}\label{SecLCK}

We recall the definition of a locally conformally K\"ahler metric:

\begin{defi}\label{lee} A Hermitian metric $g$ on a complex manifold $(M, J)$ is called {\it locally conformally K\" ahler} (lcK for short) if its fundamental form $\Omega:=g(J\cdot,\cdot)$ satisfies $d\Omega=\theta \wedge \Omega$ for a closed one-form $\theta$ on $M$. We call $\theta$ the Lee form of the metric.
\end{defi}

If $(g,\theta)$ is an lcK metric on $(M,J)$ and $f\in\ce(M,\RR)$, then $\e^fg$ is again an lcK metric with Lee form $\theta+df$. In the present section, we always assume that the Lee form is not exact. 

On the universal cover $q:\tilde M\rightarrow M$ of an lcK manifold $(M,J,g,\theta)$ we have $q^*\theta=d\varphi$. Thus, the metric $\tilde g:=\e^{-\varphi}q^*g$ becomes K\"ahler on $(\tilde M,q^*J)$ and any deck transformation $\psi \in \pi_1(M)$ acts on it by $\psi^*\tilde{g}=c_{\psi} \tilde{g}$, for some positive constant $c_{\psi}$. We call $\tilde g$ the corresponding K\"ahler metric, and note that it depends only on the conformal class of $g$. Conversely, any K\"ahler metric on $(\tilde M, q^*J)$ on which $\pi_1(M)$ acts by strict homotheties determines uniquely a conformal class of lcK metrics on $(M,J)$.

A few distinguished types of lcK metrics are of particular interest, namely the Vaisman metrics, the lcK metrics with potential and a larger class including both of them, called exact lcK metrics. An lcK metric $g$ is called {\it Vaisman} if its Lee form $\theta$ is parallel with respect to the Levi-Civita connection $\nabla^g$. It is called {\it with potential} if the corresponding K\" ahler form $\tilde \Omega$ on the universal cover admits a potential $\phi$, meaning that $\tilde \Omega=dd^c \phi$, such that $\psi^*\phi=c_{\psi}\phi$ for any $\psi \in \pi_1(M)$. Finally, it is called {\it exact} if its fundamental form $\Omega$ equals $d_{\theta}\eta$, for some one-form $\eta$, where $\theta$ is the Lee form of the metric and we denoted by $d_{\theta}:=d-\theta \wedge \cdot$. Any Vaisman metric is with potential and any lcK metric with potential is exact (see \cite{ov11}).

Brunella gives in \cite{bru} a method of constructing conformal classes of lcK metrics on every Kato surface. However, we note that his argument holds in general, for any dimension, as soon as  $\pi$ in the Kato data consists in a sequence of a blow-ups along smooth centers. In what follows, we give the outline of Brunella's proof, insisting on the arbitrariness of the dimension:

\begin{thm}\label{LCK}(see \cite[Theorem 1]{bru})
Every Kato manifold arising from a Kato data $(\pi,\sigma)$ for which $\pi$ is a sequence of blow-ups along smooth centers carries a locally conformally K\" ahler metric.
\end{thm} 

\begin{proof}
Let us denote by $M$ the $n$-dimensional Kato manifold given by the Kato data $(\pi,\sigma)$. The strategy consists in constructing a K\"ahler metric on the universal $\tilde M$ on which $\pi_1(M)$ acts by homotheties. In doing so, we use the description of $\tilde M$ given by \eqref{univc}.

To begin with, we consider a K\" ahler metric $\omega_0$ on $\hat{\DD}$, smooth up to the boundary, which exists due to the following result (see for instance \cite[Proposition~3.24]{voisin}):
\begin{prop}
If $X$ is a K\" ahler manifold and $Y \subset X$ is a compact complex submanifold, then the blown-up manifold along $Y$ is K\" ahler. 
\end{prop}
We start with modifying $\omega_0$ on $\sigma(\DD)$ such that on a neighbourhood of $\sigma(0)$, it becomes flat. The K\" ahler form $\sigma^*\omega_0$ on $\DD$ admits a smooth potential $\phi$.  We can suppose, without loss of generality, that $\phi(0)=0$ and $d_0\phi=0$. Moreover, after eliminating the pluriharmonic part of the Taylor expansion of $\phi$ at $0$, we can suppose that around $0$, $\phi$ writes: 
\begin{equation*}
\phi(z)=\frac{1}{2}\sum_{j,k=1}^n \frac{\partial^2\phi}{\del z_j \del\overline{z}_k}(0)z_j\overline{z}_k+O(||z||^3).
\end{equation*}
Note that the Hessian of $\phi$ at $0$ is positively defined since $\phi$ is a strictly plurisubharmonic function, hence $0$ is a non-degenerate local minimum of $\phi$.

We take now a positive constant $\lambda>0$ and $0< r < r'< 1$ such that the function $\rho(z):=\lambda(1+||z||^2)$ is greater than $\phi$ on $\overline{\DD}_r$ and smaller than $\phi$ on $\partial{\DD}_{r'}$. 
Indeed, there exists $r'>0$ such that $\phi>0$ on $\overline\DD_{r'}-\{0\}$. Thus we take:
\begin{equation*}
\la:=2\min_{\del \DD_{r'}}\frac{\phi}{1+{r'}^2}>0.
\end{equation*}
Moreover, as the function $\rho-\phi$ is strictly positive in $0$, it is strictly positive on $\overline\DD_r$ for some $0<r<r'$.

Following \cite[Lemma~5.18]{dem}, we can define on $\DD_{r'}$ the function $\tilde{\phi}$ as the regularized maximum $max_\epsilon(\phi, \rho)$, where $\epsilon>0$ is such that $\rho-\phi>\epsilon$ on a neighbourhood of $\partial{\DD}_r$ and $\phi-\rho>\epsilon$ on a neighbourhood of $\partial{\DD}_{r'}$. Moreover, we put $\tilde\phi=\phi$ on $\DD-\DD_{r'}$. Then $\tilde \phi$ is smooth, strictly plurisubharmonic and equal to $\rho$ on some neighbourhood of $\partial{\DD}_r$ and to $\phi$ outside $\DD_{r'}$. It follows that the \Ka\ form  $\sigma_*(dd^c\tilde{\phi})$ on $\sigma(\DD)$ glues to $\omega_0$ on $\hat\DD-\sigma(\DD)$ and defines a new smooth K\" ahler metric $\omega_1$ on $\hat{\DD}$.

The second step is to modify $\omega_1$ near the boundary of $\hat{\DD}_r=\pi^{-1}(\DD_r)$.  Consider the $(1,1)$-current on $\DD$ given by $\tilde\omega_1:=\pi_*\omega_1$. As $\DD$ is Stein, there exists a plurisubharmonic function $\psi$ on $\DD$ such that $\tilde\omega_1=dd^c\psi$. As $\tilde\omega_1$ is a smooth \Ka\ form on $\DD-\{0\}$, it follows by \cite[Corollary~3.30]{dem} that also $\psi$ is smooth on $\DD-\{0\}$. Take now $0<s<r$ such that $\sigma(\overline{\DD}_r) \subset \hat{\DD}_s$. There exist real numbers $c_1>0$ and $c_2\in \RR$ such that:
\begin{equation*}
c_1 \cdot \lambda(1+r^2) + c_2 > \max_{z \in \partial \DD_r} \psi(z)=:M, \ \ \ \ c_1 \cdot \lambda(1+s^2) + c_2 < \min_{z \in \partial \DD_s}\psi(z)=:m.
\end{equation*} 
Indeed, since $\psi$ is strictly plurisubharmonic, the function $(0,1)\ni t\mapsto \max_{\del \BB_t}\psi$ is stricty increasing by the maximum principle, thus $m<M$. It follows that for any $\al>1$, we can take:
\begin{equation}\label{Leeopen}
c_1=\al\la\frac{M-m}{r^2-s^2}>0, \ \ c_2\in \left(M-c_1\la (1+r^2),m-c_1\la (1+s^2)\right)\neq\emptyset.
\end{equation}

Thus we have $c_1\rho+c_2>\psi$ on $\partial \DD_r$ and $c_1\rho+c_2<\psi$  on $\partial \DD_s$. Let $\tilde{\psi}$ be the strictly plurisubharmonic function on $\DD$, smooth on $\DD-\{0\}$, defined as the regularized maximum of $\psi$ and of $c_1\rho+c_2$ on $\overline{\DD}_r \setminus \DD_s$, and equal to $\psi$ on $\DD_s$. Then the \Ka\ metric $\pi^*dd^c\tilde\psi$ on $\hat\BB-\hat\BB_{s/2}$ glues to $\omega_1$ on $\hat\BB_{s/2}$, defining a smooth \Ka\ metric $\omega_2$ on $\hat\BB$. Moreover, if $U$ is a small enough neighbourhood of $\del\sigma(\DD_r)$, then $V=(\sigma\circ\pi)^{-1}(U)$ is a neighbourhood of $\del\hat\BB_r$ and we have:
\begin{equation*}
(\sigma\circ\pi)^*\omega_2|_U=\frac{1}{c_1}\omega_2|_V.
\end{equation*}

Therefore we can define a smooth K\"ahler metric $\tilde\omega$ on $\tilde M=\bigsqcup_{m\in\ZZ}W_{(r),m}$ by setting $\tilde\omega:=c_1^{-m}\omega_2$ on $W_{(r),m}$, $m\in\ZZ$, where $W_{(r),m}$ is a copy of $W_{(r)}=\hat\BB_r-\sigma(\BB_r)$. Clearly the generator $\psi$ of $\pi_1(M)$ acts on $\tilde\omega$ by $\psi^*\tilde\omega=c_1^{-1}\tilde\omega$, thus $\tilde \omega$ defines a conformal class of lcK metrics on the manifold $M$.
\end{proof}

\vline

However, not all Kato manifolds admit lcK metrics. In order to construct such counter-examples, let us first recall one of Hironaka's examples \cite{hirn} of a smooth compact complex manifold bimeromorphic to a projective manifold, admitting no \Ka\ metric. Let $n\geq 3$ and consider in $\CP^n$ a curve $c$ which is smooth everywhere but at a double point $P\in c$, where it auto-intersects transversally. Let $U\subset\CC\PP^n$ be a chart around $P$ in which $(c,P)$ looks like $(\{(z_1,\ldots, z_n)\in\CC^n, z_1z_2=0, z_3=\ldots=z_n=0\},0)$ in $\CC^n$. Let $c_1, c_2$ denote the irreducible components of $c$ in this chart. Let $\mu_1:\hat U_1\rightarrow U$ denote the blow-up of $U$ along $c_1$, name $c_2'$ the strict transform of $c_2$ via $\mu_1$ and then let $\mu_2:\hat U_2\rightarrow \hat U_1$ denote the blow-up of $\hat U_1$ along $c'_2$. Also let $\mu_3=\widehat{\CC\PP^n-\{P\}}\rightarrow \CC\PP^n-\{P\}$ denote the blow-up along $c-\{P\}$. Then $\mu_3^{-1}(U-\{P\})$ is biholomorphic to $(\mu_1\circ\mu_2)^{-1}(U-\{P\})$, so we can glue $\widehat{\CC\PP^n-\{P\}}$ and $\hat U_2$ along these open sets in order to form the smooth compact complex manifold $H_n$ together with a map $\mu:H_n\rightarrow \CC\PP^n$ which coincides with $\mu_3$ or with $\mu_1\circ\mu_2$ on the appropriate open sets which cover $H_n$. Hence $\mu$ is a proper modification of $\CC\PP^n$. It turns out that $H_n$ is not projective algebraic and it admits no \Ka\ metric for $n\geq 3$. This comes from the fact that $H_n$ contains a smooth rational curve which has zero homology class, which is impossible on a compact \Ka\ manifold.

\begin{prop}\label{NOLCK}
There exist Kato manifolds which admit no lcK metric.
\end{prop}
\begin{proof}
We will construct a Kato data giving a manifold not admitting lcK metrics. 

Let us suppose $n\geq 4$. Let $\BB\subset\CC^n$ denote the standard ball and let $\pi_1:\hat\BB^{(1)}\rightarrow \BB$ be the blow-up along $0\in\BB$. Let $E_1\cong\CC\PP^{n-1}=\pi_1^{-1}(0)$ denote the exceptional divisor. Choose $c\subset E_1$ a curve with a double point $P$ like in the Hironaka construction, and let $\pi_2:\hat\BB^{(2)}\rightarrow \hat\BB^{(1)}$ denote the modification along $c$ which was described above. Namely, outside $P$, $\pi_2$ is the blow-up of $\hat\BB^{(1)}-\{P\}$ along $c-\{P\}$, while around $P$, it consists in blowing up one branch of $c$, and then the other one. Then $E_2:=\pi_2^{-1}(c)$ is a singular hypersurface $S$ which autointersects along a $\CC\PP^{n-2}$. The strict transform of $E_1$ via $\pi_2$, denoted by $E_1'$, is biholomorphic to $H_{n-1}$, the manifold given by Hironaka's example. It intersects $E_2$ along a manifold biholomorphic to $H_{n-2}$.

Choose now $\sigma:\overline{\BB}\rightarrow \BB-E'_1$ a holomorphic embedding with $\sigma(0)\in E_2$, and put $\pi:\hat\BB^{(2)}\rightarrow\BB$,  $\pi=\pi_1\circ\pi_2$. Then $(\pi,\sigma)$ is a Kato data, defining a Kato manifold $M$ of complex dimension $n$. Let $\tilde M$ denote its universal cover. By construction, we have: 
\begin{equation*}
H_{n-1}\cong E_1'\subset W=\hat\BB^{(2)}-\sigma(\overline\BB)\subset \tilde M.
\end{equation*}
Thus, $\tilde M$ cannot admit any \Ka\ metric, since such a metric would  induce one on $H_{n-1}$, which is impossible. We conclude that $M$ admits no lcK metric.  
\end{proof}



\begin{prop}\label{nopotential} There are no exact lcK metrics on a Kato manifold $M$ which is not primary Hopf. In particular, there are no lcK metrics with potential.
\end{prop}
\begin{proof} Let $C\neq\emptyset $ denote the divisor of $M$ obtained by gluing the exceptional set of $\pi$. Assume by contradiction that $\Omega=d_\theta \eta$, and take an immersion $j:\CP^1\rightarrow M$ of $\CP^1$ into an irreducible component of $C$. Then $j^*\theta=df$ for some $f\in\ce(\CP^1)$, hence the \Ka\ metric $\e^{-f}j^*\Omega$ on $\CP^1$ satisfies:
\begin{equation*}
0<\int_{\CP^1}\e^{-f}j^*\Omega=\int_{\CP^1}d(\e^{-f}j^*\eta)=0
\end{equation*}
which is impossible.
\end{proof}

\section{De Rham and twisted cohomology}\label{SecCoh}

Topologically, Kato manifolds are modifications of primary Hopf manifolds. Indeed, according to \cite[Theorem 1]{k}, for any Kato manifold $M$, there exists a complex analytical family $\Pi: \mathcal{X} \rightarrow \mathbb{D}=\{t\in\CC, |t|<1\}$ such that $\Pi^{-1}(0)=M$ and $M_t:=\Pi^{-1}(t)$ is a compact complex manifold biholomorphic to a modification of a primary Hopf manifold. 
Moreover, the manifolds $M_t$ are constructed as follows: given a Kato data $(\pi, \sigma)$ for $M$, one defines a family $\sigma_t$ depending analytically on $t$ such that $\sigma_0=\sigma$, $(\pi,\sigma_t)$ is also a Kato data and $\sigma_t(0)$ does not meet the exceptional set for $t\neq 0$. Then, the manifold $M_t$ is obtained from $(\pi, \sigma_t)$ by performing the Kato construction. In particular, when $\pi$ is a sequence of blow-ups with smooth centers, the manifolds $M_t$ are also obtained as a sequence of blow-ups with smooth centers from a primary Hopf manifold. 


On the other hand, given a compact complex manifold $M$ of dimension $n$ and $Z\subset M$ a smooth compact complex submanifold of codimension $r$, let $\pi:\hat M=Bl_ZM\rightarrow M$ be the blow-up of $M$ along $Z$ and denote by $j:E=\pi^{-1}(Z)\rightarrow \hat M$ the inclusion of the exceptional divisor. Let $j_!:H^{p-2}(E,\ZZ)\rightarrow H^{p}(\hat M,\ZZ)$ denote the Gysin morphism induced by $j$, defined as the composition of the maps:
\begin{equation*}
\xymatrix{ &H^{p-2}(E,\ZZ)\ar[r]^{PD} &H_{2n-p}(E,\ZZ) \ar[r]^{j_*} &H_{2n-p}(\hat M, \ZZ) \ar[r]^{PD} &H^p(\hat M,\ZZ) }
\end{equation*}
where $PD$ denotes the Poincar\'e duality morphism. Let us also denote by $h=c_1(\mathcal{O}_E(-1))\in H^2(E,\ZZ)$, where $\mathcal{O}_E(-1)$ denotes the tautological line bundle over $E\cong \PP(\mathcal N_MZ)$. Finally, let us introduce the map:
\begin{equation*}
\psi_q:H^{p-2q-2}(Z,\ZZ)\rightarrow H^{p}(\hat M,\ZZ), \ \  \psi_q (a)=j_!(h^q\wedge\pi|_E^*(a)).
\end{equation*} 

Then the maps $\psi_q$ and $\pi^*$ are injective, and we have the following description of the de Rham cohomology of $\hat M$ (see for instance \cite[Theorem~7.31]{voisin}):
\begin{equation}\label{coh}
H^p(\hat M,\ZZ)=\pi^*H^p(M,\ZZ)\oplus\bigoplus_{q=0}^{r-2}\psi_qH^{p-2q-2}(Z,\ZZ).
\end{equation}

In our case, we have to consider a primary Hopf manifold $M$, which is diffeomorphic to $\Ss^1\times\Ss^{2n-1}$, so that $H^0(M,\ZZ)=H^1(M,\ZZ)=H^{2n-1}(M,\ZZ)=H^{2n}(M,\ZZ)=\ZZ$ and all the other cohomology groups are $0$. Then, take $k$ consecutive blow-ups of $M$ at smooth centers $Z_p\subset E_p$,  where $E_p$ is the exceptional divisor obtained at step $p$, $1\leq p\leq k-1$, and $Z_0$ is a point. Let $\hat M$ be the resulting manifold. Using \eqref{coh} inductively, we obtain that $H^p(\hat M,\ZZ)$, as an abelian group, is completely determined by the groups $H^\bullet(Z_p,\ZZ)$. In the simplest case, we find:

\begin{cor}
Let $(\pi,\sigma)$ be a simple Kato data such that $\pi$ is a composition of $k$ blow-ups at points, and let $M=M(\pi,\sigma)$ be the corresponding $n$-dimensional Kato manifold. Then we have the following Betti numbers for $M$:
\begin{align*}
&b_{0}=b_1=b_{2n-1}=b_{2n}=1\\
&b_{2p+1}=0, \ \ &1\leq p\leq n-2\\
&b_{2p}=k, \ \ &1\leq p \leq n-1.
\end{align*}
In particular, the Euler characteristic $\chi(M)=k(n-1)$ is positive for $k>0$.
\end{cor}

Let now $\theta$ be any closed one-form on a Kato manifold $M$, and recall the differential operator 
\begin{equation*}
d_\theta\in\Omega^\bullet(M)\rightarrow \Omega^{\bullet+1}(M), \ \ d_\theta\alpha=d\alpha-\theta\wedge\alpha.
\end{equation*} 
A straightforward computation shows that $d_{\theta}^2=0$ and hence we can consider the cohomology 
\begin{equation*}
H^\bullet_{\theta}:=\frac{\mathrm{Ker}\, d_\theta}{\mathrm{Im}\, d_{\theta}}
\end{equation*}
that we shall call {\it twisted cohomology}. Via standard Hodge theory, it can be seen that these groups are finite dimensional on a compact manifold. Moreover, they only depend on the de Rham class $[\theta] \in H^1_{dR}(M)$. 

Since the twisted cohomology is a topological object, we can again use Kato's result and suppose that our Kato manifolds are modifications of primary Hopf manifolds. Then the following computes the twisted cohomology for blown-up manifolds:
\begin{thm}(\cite[Theorem 4.5]{meng})
Let $M$ be a complex $n$-dimensional manifold, let $\iota:Z \rightarrow M$ be the inclusion of a complex submanifold of codimension $r$, and let $\pi:\hat M \rightarrow M$ denote the blow-up of $M$ along $Z$. Moreover, let $\theta$ be a closed one-form on $M$. Then for any $0 \leq p \leq 2n$
\begin{equation}\label{blowup}
H^p_{\theta}(M) \oplus \bigoplus_{q=0}^{r} H_{\iota^*\theta}^{p-2-2q}(Z) \simeq H^p_{\pi^*\theta}(\hat M).
\end{equation}  
\end{thm}

In our case, the actual computation of the twisted cohomology is facilitated by the following: 
\begin{lem} The twisted cohomology of a primary Hopf manifold vanishes with respect to any closed non-exact one form. 
\end{lem}

\begin{proof} 
Any primary Hopf manifold is diffeomorphic to $M=\Ss^1 \times \Ss^{2n-1}$, and any closed one form $\theta$ on it is cohomologous to $\la dt$, where $\la\in\RR$ and $t$ is the local standard coordinate on $\Ss^1$. Now, for any metric $g$ on $\Ss^{2n-1}$, $\la dt$ is parallel with respect to the Levi-Civita connection of $g_1=dt^2+g$ on $M$. Hence, by applying \cite[Theorem~4.5]{llmp} for $\lambda dt$, with $\lambda \neq 0$, the conclusion follows. 
\end{proof}

\medskip

Thus, in order to compute the twisted cohomology of a  Kato manifold given by a sequence of blow-ups with smooth centers, we can use the above two results inductively, and as in the case of de Rham cohomology, we find that $H^\bullet_\theta(\hat M)$ is completely determined by the groups $H^\bullet_{\iota^*\theta}(Z_p)$, where $Z_p$ are the smooth centers at which the blow-ups were performed. In the simplest case we obtain:

\begin{cor}
Let $M$ be a Kato manifold of dimension $n$ with simple Kato data $(\pi,\sigma)$, such that $\pi$ is a composition of $k$ blow-ups at points. Also let $\theta$ be a closed non-exact one-form on $M$. Then we have the following twisted Betti numbers:
\begin{align*}
&\dim H^{2p}_{\theta}(M)=k, \ \ &1\leq p \leq n-1\\
&\dim H^{p}_{\theta}(M)=0, \ \ &p\in\{2q+1, q\in\ZZ\}\cup\{0, 2n\}.
\end{align*}
\end{cor}

\section{Kato manifolds generalizing Inoue-Hirzebruch surfaces}\label{SecIH}

\subsection{Kato matrices} In this section, we introduce a class of matrices that we will use in the sequel to construct a certain class of Kato manifolds. We also discuss some of their properties that we will need later.

Let $A=(a_{pq})_{1\leq p\leq m,1\leq q\leq n}$ be an $m\times n$ matrix with entries that are natural numbers. We introduce the following notation, which we shall use throughout the paper:
\begin{equation*}
z^A:=(z_1^{a_{11}}z_2^{a_{12}}\ldots z_n^{a_{1n}}, \ldots,z_1^{a_{m1}}z_2^{a_{m2}}\ldots z_n^{a_{mn}}), \ \ \ z\in\CC^n.
\end{equation*}
Moreover, we will denote by $F_A:\CC^n\rightarrow\CC^m$ the holomorphic map $F_A(z)=z^A$. Note that for any two $n\times n$ matrices $A$, $B$ one has the relation $F_{AB}=F_A\circ F_B$. In particular, if $A\in\GL(n,\ZZ)$, then $F_A\in\Aut(\Cn{n})$.

Let now $e_1,\ldots, e_n$ denote the standard basis of $\RR^n$, written as $n\times 1$ matrices, let $c=\sum_{j=1}^ne_j$ and consider  the $n\times n$ {\it elementary matrices} defined by:
\begin{equation}\label{elmatrices}
A_j= \begin{pmatrix}e_1 & \cdots & e_{j-1} & e_{j+1} & \cdots & e_n & c
\end{pmatrix}, \ \ 1\leq j\leq n. 
\end{equation}
Written on components, these are:
\begin{equation*}
A_{1} = 
 \begin{pmatrix}
  0 & 0 & 0 & \cdots & 0 & 1 \\
  1 & 0 & 0 & \cdots & 0 & 1 \\
  0 & 1 & 0 & \cdots & 0 & 1 \\
  \vdots  & \vdots & \vdots  & \ddots & \vdots &  \vdots  \\
  0 & 0 & 0 & \cdots & 1 & 1 
 \end{pmatrix}
 \ \
A_{2} = 
 \begin{pmatrix}
  1 & 0 & 0 & \cdots & 0 & 1 \\
  0 & 0 & 0 & \cdots & 0 & 1 \\
  0 & 1 & 0 & \cdots & 0 & 1 \\
  \vdots  & \vdots & \vdots  & \ddots & \vdots &  \vdots  \\
  0 & 0 & 0 & \cdots & 1 & 1 
 \end{pmatrix}
\
\ldots
\
A_{n} = 
 \begin{pmatrix}
  1 & 0 & 0 & \cdots & 0 & 1 \\
  0 & 1 & 0 & \cdots & 0 & 1 \\
  0 & 0 & 1 & \cdots & 0 & 1 \\
  \vdots  & \vdots & \vdots  & \ddots & \vdots &  \vdots  \\
  0 & 0 & 0 & \cdots & 0 & 1 
 \end{pmatrix}.
\end{equation*}

We are interested in products of such matrices. We will call a product of elementary matrices $A=A_{j_1}\cdots A_{j_k}$, where repetitions are allowed, a \textit{Kato matrix} if $A\neq A_n^p$ for some $p\geq 1$. These matrices were first considered by Dloussky in \cite{d1} in the case $n=2$.

\begin{lem}\label{contractie}
Given a product of elementary matrices $A=A_{j_1}\cdots A_{j_k}$, with $1\leq j_1,\ldots, j_k\leq n$, the map $F_A$ satisfies $F_A(\overline{\DD})\subset\DD$ if and only if $A$ is not of the form $A=A_qA_n^{p}$, for $1\leq q\leq n$ and $p\geq 0$.
\end{lem}
\begin{proof}
Suppose $A=(a_{st})_{s,t}$ does not satisfy $F_A(\overline{\DD})\subset\DD$.
Since for any $1\leq s \leq n$ we have $\sum_{t=1}^na_{st}\geq 1$, it follows that $||F_A(z)||\leq ||z||$ each time $z\in\DD$. Thus $F_A(\overline{\DD})\not\subset \DD$  if and only if there exists $w\in\del\DD$ with $||F_A(w)||=1$.

Now for an elementary matrix $A_p$, $1\leq p\leq n$, $||F_{A_p}(w)||=1$ reads:
\begin{equation*}
1=|w_n|^2(1+\sum_{s=1}^{n-1}|w_s|^2)=|w_n|^2(2-|w_n|^2)
\end{equation*}
i.e. $|w_n|=1$ and $w_1=\ldots=w_{n-1}=0$. Thus, if we let $d=\{z\in\CC^n, z_1=\ldots=z_{n-1}=0\}$, since $F_A$ does not contract the ball, it follows that either $k=1$ or that $k>1$ and $F_{A_q}(d)=d$ for any $q\in\{j_2,\ldots, j_k\}$. On the other hand, we have $F_{A_q}(d)=d$ if and only if $q=n$. Hence, if $k>1$ then $A_{j_2}=\ldots=A_{j_k}=A_n$. 

Conversely, note that if we denote by $e_1,\ldots, e_n\in\CC^n$ the standard basis, then we have:
\begin{equation*}
F_{A_qA_n^p}(e_n)=e_q, \ \ \  1\leq q\leq n, \ p\geq 0.
\end{equation*}
Thus indeed the matrix $A_qA_n^p$ does not contract $\BB$.
\end{proof}

Consider now the ball for the second norm $\BB_{1,2}:=\{z\in\CC^n, ||z||^2_{1,2}=\sum_{j=1}^{n-1}|z_j|^2+2|z_n|^2<1\}$. Note that $\BB_{1,2}$ is biholomorphic to $\BB$ via $(z_1,\ldots, z_n)\mapsto (z_1,\ldots, z_{n-1},\sqrt{2}z_n)$. As it turns out, all Kato matrices define germs which contract this second ball:

\begin{lem}\label{contractie12}
If $A\in\GL(n,\ZZ)$ is a Kato matrix, then $F_A(\overline{\BB}_{1,2})\subset\BB_{1,2}$.
\end{lem} 
\begin{proof}
Let $z\in\overline{\BB}_{1,2}$. Suppose first that $A=A_q$ with $1\leq q\leq n-1$. Then we find:
\begin{align*}
||F_{A_q}(z)||^2_{1,2}&=|z_n|^2(1+\sum_{j=1}^{n-2}|z_j|^2+2|z_{n-1}|^2)\\
&<|z_n|^22(1+\sum_{j=1}^{n-1}|z_j|^2)\\
&\leq (1-\sum_{j=1}^{n-1}|z_j|^2)(1+\sum_{j=1}^{n-1}|z_j|^2)\leq 1
\end{align*}
hence $F_{A_q}(z)\in\BB_{1,2}$. Moreover, in the same way we find $F_{A_n}(z)\in\overline{\BB}_{1,2}$.

Now if $A=A_{j_1}\cdots A_{j_k}$ is Kato, then there exists $1\leq s\leq k$ with $j_s\neq n$. Hence, applying inductively the result for elementary matrices, we find $F_A(z)\in\BB_{1,2}$.
 \end{proof}

\begin{rmk}
As was noted above, the line $d=\{z_1=\ldots=z_{n-1}=0\}\subset\CC^n$ is fixed pointwise by $F_{A_n}$. Hence, for $A=A_n^p$, $p\geq 1$, there exists no open set $D\subset\CC^n$ containing $0$  with $F_A(\overline D)\subset D$.
\end{rmk}

For $n=2$, Kato matrices are known to have a unique factorisation into elementary matrices. Our next goal is to show that this still holds for $n\geq 3$. To this aim, we will first introduce some notation.

We start by defining a relation on elements of $\ZZ^n$. For two vectors $v,w\in\ZZ^n$, we say that $v\prec w$ if either $v=e_i$, $w=e_j$ and $i<j$, or $v_i\leq w_i$ for all $1\leq i\leq n$, with at least one inequality being strict. Note that this relation is not transitive. For $A=(a_{ij})_{ij}\in\GL(n,\ZZ)$, denote by $A^{(j)}=(a_{1j}, \ldots, a_{nj})^t$ the $j$-th column of $A$, for $1\leq j\leq n$. 

\begin{lem}\label{ordinecoloane}  Let $A$ be a product of elementary matrices defined in \eqref{elmatrices}. Then:
\begin{equation}\label{proprO} 
A^{(i)}\prec A^{(j)} \ \ \text{ for any }1\leq i<j\leq n.
\end{equation}
\end{lem}

\begin{proof}
It is clear by \eqref{elmatrices} that the elementary matrices satisfy \eqref{proprO}. Then one notices that for $B\in\GL(n,\ZZ)$ and $1\leq j\leq n$, one has:
\begin{equation}\label{productK}
B\cdot A_j=\begin{pmatrix} B^{(1)} & \cdots & B^{(j-1)} & B^{(j+1)} &\cdots & B^{(n)} & \sum_{i=1}^n B^{(i)} \end{pmatrix}.
\end{equation}
Thus, if $B$ is a product of elementary matrices and satisfies \eqref{proprO}, then clearly also $B\cdot A_j$ satisfies this property, which concludes the proof.
\end{proof}

\begin{lem}\label{unicFact} If $A=A_{j_1} \cdot \ldots \cdot A_{j_k}=A_{i_1} \cdot \ldots \cdot A_{i_p}\in\GL(n,\ZZ)$, then $p=k$ and $j_s=i_s$, for any $s \in \{1, \ldots, k\}$. 
\end{lem}

\begin{proof} 
We will give an algorithm of factorisation of $A$ into elementary matrices which will show that the factorisation is unique.

Let us write $A=B\cdot A_j$, where $B$ is a product of elementary matrices and $1\leq j\leq n$. We note that by \eqref{productK}, the columns of $A$ completely determine those of $B$ as being $A^{(1)},\ldots, A^{(n-1)}$, $C:=A^{(n)}-\sum_{i=1}^{n-1}A^{(i)}$, not necessarily in this order. Moreover, by Lemma~\ref{ordinecoloane}, there exists only one possible order of these columns in $B$, hence $B$ is uniquely determined by $A$.  Finally, $j$ clearly satisfies $A^{(j-1)}\prec C\prec A^{(j)}$, hence it is also uniquely determined by $A$. Thus the decomposition $A=B\cdot A_j$ is unique for $A$, allowing us to uniquely factorise $A$ into elementary matrices.  
\end{proof}

Next, since we have:
\begin{equation*}
\begin{pmatrix}
  I_l & *  \\
  0 & B
  \end{pmatrix}
  \begin{pmatrix}
  I_l & * \\
  0 & D
 \end{pmatrix}
 = \begin{pmatrix}
  I_l & * \\
  0 & BD
 \end{pmatrix}
\end{equation*}
for any Kato matrix $A$ there exists a maximum number $l=l(A)\geq 0$ such that $A$ is of the form:
\begin{equation}
 \begin{pmatrix}
  I_l & G \\
  0 & B
 \end{pmatrix},
 \end{equation}
 where $B\in\GL(n-l,\ZZ)$ is a Kato matrix. Note that, if $A=A_{j_1}\cdots A_{j_k}$, then
 \begin{equation}
l(A)=\min_{1\leq p\leq k}j_p-1.
\end{equation}
In particular, since $A$ is a Kato matrix, it follows that $l\leq n-2$. We will say that $A$ is \textit{of type $l$}; in this case, $A^m$ is also of type $l$, for any $m>0$.

\begin{rmk}\label{eigenv1}
Any type $l$ Kato matrix with $l \geq 1$ has $1$ as eigenvalue of geometric multiplicity at least $l$.
\end{rmk}

\begin{lem}\label{LineL}
Any type $l$ Kato matrix $A\in\GL(n,\ZZ)$ is of the form $A=\begin{tiny}
\begin{pmatrix}
  I_l & G \\
  0 & B
 \end{pmatrix}
\end{tiny}$, such that all the lines of $G$ are equal to $L=(p_{l+1},\ldots, p_n)\in\NN^{n-l}$. In particular, if we denote by $z=(z_1,\ldots, z_l)$ the first $l$ coordinates on $\CC^n$ and  by $w=(w_{l+1},\ldots, w_n)$ the last $n-l$ coordinates, then we have:
\begin{equation}\label{standardform}
F_A(z,w)=(z_1w^L,\ldots, z_lw^L, F_B(w)).
\end{equation}
where $w^L=w_{l+1}^{p_{l+1}}\cdots w_n^{p_n}$. 
 \end{lem}
 \begin{proof}
 Let $A=A_{j_1}\cdots A_{j_k}$, with $j_s>l$ for any $s\in\{1,\ldots ,k\}$. We will show by induction on $k$ that all the lines of $G$ are equal. For $k=1$, we have by definition (cf. \eqref{elmatrices}) that all the lines of $G$ equal $L_0:=(0,\ldots, 0,1)$. Now suppose that $A'=A_{j_2}\cdots A_{j_k}=\begin{tiny}
\begin{pmatrix}
  I_l & G' \\
  0 & B'
 \end{pmatrix}
\end{tiny}$ and all the lines in $G'$ are equal to $L'$, and let us moreover write $A_{j_1}=\begin{tiny}
\begin{pmatrix}
  I_l & G'' \\
  0 & B''
 \end{pmatrix}
\end{tiny}$, where all the lines in $G''$ equal $L_0$. Then we find:
\begin{equation*} 
G=G'+G''B'
\end{equation*}
hence any line of $G$ equals $L'+$ the last line of $B'$, and the conclusion follows.  \end{proof}

\begin{lem}\label{vectprJ}
Let $A=\begin{tiny}
\begin{pmatrix}
  I_l & G \\
  0 & B
 \end{pmatrix}
\end{tiny}\in\GL(n,\ZZ)$ be a type $l>0$ Kato matrix and let $L$ denote a line of $G$. Then $J_0=(1,\ldots, 1)\in\ZZ^{n-l}$ satisfies 
\begin{equation}\label{J_0l}
J_0B-J_0=(n-l-1)L.
\end{equation}
In particular, if for any $M\in\GL(j,\ZZ)$, $l_M:\ZZ^j\rightarrow \ZZ^j$ denotes the map $v\mapsto vM$ and $K_M:=\ker(l_M-\id)$, then we find that
\begin{equation*}
\Theta:=\{(I,J)\in\ZZ^l\times\ZZ^{n-l},\exists a\in\ZZ\ \exists K\in K_B \ \sum_{p=1}^lI_p=a(n-l-1), J=K-aJ_0\}
\end{equation*}
is of finite index in $K_A$ and $\rank K_B=\rank K_A-l$.
\end{lem}
\begin{proof}
First, we will show that eq. \eqref{J_0l} holds for any product of the form:
\begin{equation}\label{prodl} 
A=A_{j_1}\cdots A_{j_k}=\begin{tiny}
\begin{pmatrix}
  I_l & G \\
  0 & B
 \end{pmatrix}
\end{tiny}, \ \min_{1\leq p\leq k}j_p>l
\end{equation} 
by induction on $k\geq 1$. For $k=1$, $A=A_j$ with $j>l$, $L=L_0=(0,\ldots, 0,1)$ and
\begin{equation}\label{Bj}
B=B_j:=\begin{pmatrix}
e_{l+1} & \cdots & e_{l+j-1} & e_{l+j+1} & \cdots & e_n & c
\end{pmatrix}
\end{equation}
where $e_{l+1}\ldots e_{n}$ denotes the standard basis of $\RR^{n-l}$ and $c=\sum_{s=l+1}^{n}e_s$, thus we find \eqref{J_0l}. 

Suppose now that we have $A=A_jA'$, $j>l$, where $A'=\begin{tiny}
\begin{pmatrix}
  I_l & G' \\
  0 & B'
 \end{pmatrix}
\end{tiny}$ is a product as in \eqref{prodl}, with $L'$ denoting a line of the matrix $G'$, and suppose that \eqref{J_0l} is satisfied for $A'$. It follows that $B=B_jB'$ and $L=L'+B'_{(n)}$, where $B'_{(n)}$ denotes the last line of $B'$. Thus we find, using \eqref{Bj} and the induction hypothesis:
\begin{align*}
J_0B=(J_0+(n-l-1)L_0)B'=J_0+(n-l-1)L'+(n-l-1)B'_{(n)}=J_0+(n-l-1)L
\end{align*}
which is precisely what we wanted to show.

Next, using Lemma~\ref{LineL}, we find that $(I,J)\in\ZZ^l\times\ZZ^{n-l}$ is in $K_A$ if and only if $(\sum_{p=1}^l I_p)L+JB=J$. Thus we infer by \eqref{J_0l} that we have $\Theta\subset K_A$. Moreover: 
\begin{equation}\label{dimTheta}
\dim \Theta\otimes\QQ=l+\dim K_B\otimes\QQ\leq \dim K_A\otimes\QQ.
\end{equation}

On the other hand, we also have $\dim K_A\otimes\QQ=\dim\ker p_A$, where for any $M\in\GL(j,\ZZ)$ we put $p_M:\QQ^j\rightarrow \QQ^j$, $v\mapsto v^t-Mv^t$. But $\ker p_A=\QQ^l\oplus(\ker p_B\cap \{v\in\QQ^{n-l}, Lv^t=0\})$, so:
\begin{equation}\label{dimP_B} 
\dim K_B\otimes\QQ=\dim\ker p_B=\dim\ker p_A-l+\delta,
\end{equation} 
with $\delta\in\{0,1\}$, depending on whether or not $\ker p_B$ is included in the hyperplane $\{v\in\QQ^{n-l}, Lv^t=0\}$. Combining \eqref{dimTheta} with \eqref{dimP_B}, we find $\delta=0$ and $\dim\Theta\otimes\QQ=\dim K_A\otimes\QQ$. In particular, $\Theta$ is of finite index in $K_A$.
\end{proof}

\begin{rmk}\label{vectprJn-2}
If $l=n-2$, then we find in the above lemma that $J_0B-J_0=L$. Moreover, $K_B=0$ in this case. Thus we have:
\begin{equation}
\Theta=\{(I,-(\sum_{p=1}^lI_p)J_0)\in\ZZ^l\times\ZZ^{n-l}, I\in\ZZ^l\}=K_A.
\end{equation}
More generally, in the case $\rank K_B=0$ we find $\Theta=K_A$.
\end{rmk}

\begin{defi} We shall call a type $l$ Kato matrix {\it $l$-positive} if it is of the form:
\begin{equation*}
\begin{pmatrix}
  I_l & * \\
  0 & B
 \end{pmatrix}
\end{equation*}
and all the entries of $B$ are strictly positive numbers. If $l=0$, we will simply call it \textit{positive}.
\end{defi}

\begin{rmk}\label{positive} It can be easily shown that for any Kato matrix $A\in\GL(n,\ZZ)$ of type $l$, $A^{p}$ is $l$-positive for any $p\geq n-l$.
\end{rmk}
\subsection{The manifolds $M_A$} In this section, we will associate to any Kato matrix $A\in\GL(n,\ZZ)$ an $n$-dimensional Kato manifold, which we will denote by $M_A$.

Recall that the blow-up of $\CC^n$ at $0\in\CC^n$ is defined as: 
\begin{equation*}
\hat{\mathbb{C}}^n=\{((z_1,\ldots, z_n), [u_1: u_2: \ldots :u_n])  \mid z_lu_j=z_ju_l, 1\leq j,l\leq n\} \subseteq \mathbb{C}^n \times \mathbb{C}\mathbb{P}^{n-1}
\end{equation*}
and we denote by $\Pi:\hat\CC^n\rightarrow \CC^n$ the natural projection. We can cover $\hat\CC^n$ by the standard charts $f_j: \mathbb{C}^n \rightarrow \hat{\mathbb{C}}^n$,
\begin{equation*}
f_j(z)=((z_1z_n, \ldots, z_{j-1}z_n, z_n, z_{j}z_n, \ldots, z_{n-1}z_n ), [z_1: \ldots :z_{j-1}:1:z_{j}:\ldots:z_{n-1}])
\end{equation*}
for $1\leq j\leq n$. In these charts, the projection is expressed as $\Pi\circ f_j(z)=z^{A_j}$, where $A_j$ are given in \eqref{elmatrices}. As a matter of fact, this unravels the motivation for considering the matrices $A_j$. Note moreover that in any of the charts, the exceptional divisor is given by the equation $z_n=0$. 

Given a Kato matrix $A$, by Lemma~\ref{unicFact} we can uniquely write it as  $A=A_{j_1}\cdots A_{j_k}$, $k\geq 1$. We will construct a Kato data $(\pi,\sigma)$ by the composition procedure presented in Section~\ref{SecKato} of the data $(\Pi,f_{j_1}),\ldots, (\Pi,f_{j_k})$. Let $\D:=\BB_{1,2}=\{z\in\CC^n, \sum_{j=1}^{n-1}|z_j|^2+2|z_n|^2<1\}$. By Lemma~\ref{contractie12}, we have $F_{A_j}(\D)\subset \D$ for any $1\leq j\leq n$. 

Denote also by $\Pi:\hat \D\rightarrow \D$ the blow-up of $\D\subset\CC^n$ at $0$. For $1\leq j\leq n$, let $D_j:=f_{j}(\D)$ and $W_j:=\hat\D-D_j$. Note that the map 
\begin{equation*}
\gamma_j=f_j\circ\Pi|_{\del \hat\D}:\del\hat\D\rightarrow \del D_j
\end{equation*}
extends to a biholomorphism from a neighborhood of $\del\hat\D$ in $\hat\CC^n$ to a neighbourhood of $\del D_j$ in $\hat\CC^n$. We then use these maps to define, by gluing, the following complex manifolds, for $1\leq p\leq k$:
\begin{align}\label{defW_l}
\begin{split}
\hat\D^{(p)}&:=W_{j_1}\sqcup_{\gamma_{j_1}}W_{j_2}\sqcup_{\gamma_{j_2}}\ldots \sqcup_{\gamma_{j_{p-1}}}\hat \D \\ 
W^{(p)}&:=W_{j_1}\sqcup_{\gamma_{j_1}}W_{j_2}\sqcup_{\gamma_{j_2}}\ldots \sqcup_{\gamma_{j_{p-1}}} W_{j_p}.
\end{split}
\end{align}
If we denote by $\iota_p:\hat\D\rightarrow\hat\D^{(p)}$ the natural embedding and let $\sigma_p=\iota_p\circ f_{j_p}:\D\rightarrow \hat\D^{(p)}$, then $W^{(p)}=\hat\D^{(p)}-\sigma_p(\D)$. Note that, for $p\geq 1$, $\hat\D^{(p)}$ is the blow-up of $\hat\D^{(p-1)}$ at $\sigma_{p-1}(0)$, and we denote by $\pi_p:\hat\D^{(p)}\rightarrow \hat\D^{(p-1)}$ this blow-up map. Here, by convention, $\hat\D^{(0)}=\D$ and $\sigma_0=\id$. Moreover, we have a commutative diagram:
\begin{equation}\label{sigma}
\xymatrix{ &\hat\D\ar[r]^-{\iota_p} \ar[d]^-\Pi &\hat\D^{(p)}\ar[d]^-{\pi_p}\\
&\D \ar[r]^-{\sigma_{p-1}} & \hat\D^{(p-1)}
}
\end{equation}

Finally, let $\pi:=\pi_1\circ\ldots\circ\pi_k:\hat\D^{(k)}\rightarrow \D$, with $E:=\pi^{-1}(0)$ denoting the exceptional divisor of $\pi$, and let $\sigma:=\sigma_k$, so that $W=\hat\D^{(k)}-\sigma(\D)=W^{(k)}$. Then we find inductively, via \eqref{sigma} and using that $\Pi\circ f_{j}=F_{A_j}$:
\begin{align*}
\pi\circ\sigma&=\pi_1\circ\ldots\circ\pi_{k-1}\circ \sigma_{k-1}\circ F_{A_{j_k}}\\
&=F_{A_{j_1}}\circ\ldots \circ F_{A_{j_k}}\\
&=F_A. 
\end{align*} 

By Lemma~\ref{contractie12}, we find $\sigma(\overline\D)\subset\hat\D^{(k)}$. Therefore, the data $\pi:\hat\D^{(k)}\rightarrow\D$ and $\sigma:\overline{\D}\rightarrow \hat\D^{(k)}$ defines a Kato data and the manifold:
\begin{equation*}
M_A:=M(\pi,\sigma)
\end{equation*}
is a Kato manifold.

\begin{rmk}
If $A\neq A_qA_n^p$ with $1\leq q\leq n-1$, $p\geq 0$, then by Lemma~\ref{contractie} we can take $D:=\BB$ instead of $\BB_{1,2}$ in the above construction of the Kato data. In this way, the resulting data $(\pi,\sigma)$ is a Kato data in the classical sense, meaning that $\pi$ is a modification of the ball $\BB$. 
\end{rmk}

Lemma~\ref{ciclic} immediately translates to:
\begin{lem} Let $A=A_{j_1}\cdots A_{j_{k-1}}A_{j_k}$ be a Kato matrix and let $A^c=A_{j_k}A_{j_1}\cdots A_{j_{k-1}}$ be the Kato matrix obtained by a cyclic permutation of the elementary factors of $A$. Then $M_A$ is biholomorphic to $M_{A^c}$.   
\end{lem}

Similarly, Lemma~\ref{finitecov} reads:
\begin{lem}\label{cov}
Let $A$ be a Kato matrix and let $p\geq 2$. Then $M_{A^p}$ is a finite cyclic unramified covering of $M_A$ with $p$ sheets. 
\end{lem}

\begin{rmk}\label{subvarKato}
Let us note that if $A=\begin{tiny}
\begin{pmatrix}
  I_l & G \\
  0 & B
 \end{pmatrix}
\end{tiny}$ is a type $l$ Kato matrix, such that $F_A(z,w)=(z_1w^L,\ldots, z_lw^L, F_B(w))$, $L\in\NN^{n-l}$, then $M_A$ contains proper Kato submanifolds. Indeed, for any $1\leq r\leq l$ and $1\leq j_1<\ldots<j_r\leq l$, the $r$-codimensional submanifold $D':=\{z_{j_1}=0\}\cap\ldots \cap\{z_{j_r}=0\}\cap\D$ in $\D$ is invariant under $F_A$, thus $\pi^{-1}(D')\cap W$ is $\gamma$-invariant and defines an $r$-codimensional submanifold of $M_A$. Moreover, this manifold is biholomorphic to the Kato manifold $M_{A'}$, where $A'=\begin{tiny}
\begin{pmatrix}
  I_{l-r} & G' \\
  0 & B
 \end{pmatrix}
\end{tiny}\in\GL(n-r,\ZZ)$ and the matrix $G'$ is the matrix having $l-r$ lines, all equal to $L$. 
\end{rmk}

Let $\tilde{C}$ be the string of divisors in $\tilde{M}_A$ obtained by gluing the corresponding traces on $W_m$ of the exceptional divisor $E$ of $\hat\D^{(k)}$. As in the proof of \cite[Proposition~1.11, Part~I]{dl}, it can be seen that $\tilde C$ is a formal sum of all the immersed irreducible compact hypersurfaces in $\tilde M_A$.  Furthermore, we denote by $C=q(\tilde C)$ the corresponding cycle of divisors on $M$. Then $C$ has $k$ irreducible components $C_1,\ldots, C_k$, each coming from an irreducible component of $E$. Moreover, for each $1\leq j\leq k$, $C_j$ is a proper modification of $\CC\PP^{n-1}$ at finitely many points. In particular, $C_j$ is a rational manifold, in the sense that it is birational to $\CC\PP^{n-1}$.

\begin{lem}\label{Hj} Let $A=A_{j_1}\cdots A_{j_k}$ be a Kato matrix of type $l$, let $(\pi,\sigma)$ be the Kato data that it determines and let $H=\{z_{l+1}\cdots z_n=0\}\subset\D$. Then $\pi^{-1}(H)\cap W=C\cap W$, hence the divisor in $M_A$ induced by gluing $\pi^*H$ is precisely $C$.  Moreover, if $A$ is $l$-positive, then the strict transforms $H_j'$ of $H_j=\{z_j=0\}\subset\D$ via $\pi$, for $l+1\leq j\leq n$, determine $n-l$ distinct irreducible components $C_{p_{l+1}},\ldots, C_{p_{n}}$ of $C$. In particular, $k\geq n-l$.
\end{lem}
\begin{proof}
Let $E=\sum_{j=1}^kE_j$ denote the exceptional divisor of $\pi$, so that $\pi^*H=E+\sum_{j=l+1}^nH'_j$.  Since $A$ is of type $l$, we have $F_A(H_j)\subset H_j$ for any $1\leq j \leq l$. Fix now $l+1\leq j\leq n$. Then we claim that  $F_A(H_j)\subset H_s$ with $s>j$. Indeed, this is clear by looking at elementary matrices: $F_{A_p}(H_j)\subset H_s$, where $s=j$ if $p\geq j+1$ and $s=j+1$ if $p\leq j$. Moreover, since $A$ is of type $l$, there exists at least one factor $A_{l+1}$ in the factorisation of $A$, hence the claim follows.
In addition, clearly $F_A(H_n)=\{0\}$. 

Let $F$ be an irreducible component of $E$ which intersects $\sigma(\D)$. Then $\sigma^{-1}(F)=H_j$ with $l+1\leq j\leq n$. Next, either $H_j'$ does not intersect $\sigma(\D)$, either it does, in which case $\sigma^{-1}(H_j')=H_{\nu_1}$. By our claim we must have $l+1\leq \nu_1<j$. Thus, we find a finite sequence of natural numbers $l+1\leq\nu_m<\ldots<\nu_1<\nu_0=j$ so that $H'_{\nu_m}\cap\sigma(\D)=\emptyset$ and $\sigma^{-1}(H'_{\nu_s})=H_{\nu_{s+1}}$ for $0\leq s\leq m-1$. It follows that $F$ determines an irreducible component of $C$ given by 
\begin{equation*}
F\sqcup_\gamma H'_{\nu_0}\sqcup_\gamma\ldots\sqcup_\gamma H'_{\nu_m}.
\end{equation*}
This shows that $C\cap W\subset \pi^* H\cap W$. 

Conversely, given any $j$ with $l+1\leq j\leq n$, our claim implies the existence of a sequence of natural numbers $j=\nu_0<\nu_1<\ldots<\nu_p\leq n$, $p\geq 0$, satisfying 
\begin{equation*}
0\neq F_A(H_{\nu_s})\subset H_{\nu_{s+1}}, \ \ 0\leq s\leq p-1, \  F_A(H_{\nu_p})=\{0\}.
\end{equation*}
But now, since $F_A=\pi\circ\sigma$, this implies that $\sigma(H_{\nu_p})\subset E_s$ for some $1\leq s\leq k$. Thus $H'_j$ glues, via $\gamma$, to $E_s$ to give a component of $C$.  This shows $\pi^*H\cap W\subset C\cap W$.
   
Finally, in the case when $A$ is $l$-positive, $F_A(H_j)=\{0\}$ for any $l+1\leq j \leq n$. This implies that $H'_j\cap\sigma(\D)=\emptyset$ and that $\sigma(H_j)\subset E_{p_j}$ for some $1\leq p_j\leq k$. Thus to each such $j$ corresponds an irreducible component of $C$ given by:
\begin{equation*}
C_{p_j}=H_j'\sqcup_\gamma E_{p_j}
\end{equation*}
and so $k\geq n-l$.
\end{proof}

In complex dimension $2$, $\tilde C$ always has two connected components, as shown in \cite[Corollary 3.29, Part~I]{dl}. Moreover, $C$ has two connected components if $M_A$ is {\it even} and only one if it is {\it odd}, by \cite{n84}.  In higher dimension, we have the following:

\begin{lem}
Let $A\in\GL(n,\ZZ)$ be a Kato matrix and suppose $n>2$. Then $\tilde C$ and $C$ are connected.
\end{lem}
\begin{proof}
Up to taking a positive power of $A$, which does not change $\tilde C$ by Lemma~\ref{cov}, we can suppose that $A$ is $l$-positive.

By Lemma~\ref{Hj}, we can enumerate the components of the exceptional divisor $E$ of $\pi$, $E_1,\ldots, E_k$, such that $E_1,\ldots, E_{n-l}$ intersect $\sigma(\D)$ and $E_{n-l+1},\ldots, E_k\subset W$. Moreover, there exists a bijection $\nu:\{1,\ldots, n-l\}\rightarrow\{l+1,\ldots, n\}$ such that $\sigma(H_{\nu(j)})\subset {E_j}$, for $1\leq j\leq n-l$. 

For each $m\in\ZZ$, let $W_m$ be a copy of $W$. Let us also denote by $E_{j,(m)}$ and by $H'_{j,(m)}$ the copies of $H'_j$ and of $E_j$ in $W_m$. For $1\leq j\leq n-l$, we put:
\begin{equation*}
\tilde C^{m}_j:=H'_{\nu(j),(m-1)}\sqcup_\gamma E_{j,(m)}\subset W_{m-1}\sqcup_\gamma W_{m}.
\end{equation*} 
For $n-l+1\leq j\leq k$, we put $\tilde C_j^{m}:=E_{j,(m)}\subset W_{m}$. Thus $\tilde C^{m}:=\sum_{j=1}^k\tilde C_j^{m}$ is a lift of $C$ to $\tilde M_A$ and $\tilde C=\cup_{m\in\ZZ}\tilde C^{m}$.

Now, since $n>2$ and we only perform blow-ups at points, we have $\emptyset\neq H'_{\nu(j)}\cap H'_{\nu(p)}\subset W$ for any $1\leq j,p\leq n-l$, hence $\tilde C^m\cap W_{m-1}$ is connected. On the other hand, $\tilde C^m$ is obtained by cutting a connected subset from $E$, $E\cap\sigma(D)$, and gluing back another connected subset, $\sum_{j=l+1}^kH'_{j,(m-1)}$. We infer that since $E$ is connected, also $\tilde C^m$ is connected. 

Finally, since for any $l+1\leq j\leq k$,  $H'_j$ intersects $E$ in $W$, it follows that $\tilde C^{m+1}$ intersects $\tilde C^m$ in $W_m$. We thus conclude that $\tilde C$ is connected and so is $C=q(\tilde C)$. 
\end{proof}

\section{The compactification point of view}\label{secComp}

In complex dimension 2, the manifolds $M_A$ are Inoue-Hirzebruch surfaces (see \cite{d1}, \cite{d2}). In a description coming from number theory, due to Hirzebruch (see \cite{h}), these surfaces appear as compactifications of  quotients of $\mathbb{H} \times \mathbb{C}$ by $U \ltimes L$, where $L$ is a finite index lattice in the ring of integers of some quadratic field $K$ and $U$ is a cyclic group of positive units in $K$. 

In what follows, we wish to show that a similar compactification statement holds in any dimension.  Let us first introduce one more definition.

Given a holomorphic map $F:\CC^n\rightarrow \CC^n$ which fixes $0\in\CC^n$, one can define its stable set:
\begin{equation*}
W^s(F):=\{z\in\CC^n|\lim_{m\to\infty}||F^{m}(z)||=0\}.
\end{equation*}
Note that for any $p>0$ one has $F(W^s(F^{p}))\subset W^s(F^{p})$, from which it easily follows that $W^s(F)=W^s(F^{p})$. For a Kato matrix $A\in\GL(n,\ZZ)$ of type $l$, let us moreover define:
\begin{equation*}
W^s(F_A)^*:=W^s(F_A)\cap(\CC^l \times \Cn{n-l}) \subset \CC^l \times \Cn{n-l}. 
\end{equation*}
Note that $W^s(F_A)^*$ is fixed by the action of the group $U_A:=\langle F_A\rangle\cong\ZZ$. In all that follows, for fixed $l\geq 0$, we will use the notation:
\begin{equation*}
\DD^*=\DD\cap(\CC^l \times \Cn{n-l})\subset \CC^n.
\end{equation*}

\begin{lem}\label{descrW}
Given a Kato matrix $A$ of type $l$, one has $W^s(F_A)^*=\bigcup_{m\in\ZZ} F_A^m(\DD^*)$.
\end{lem}
\begin{proof}
Let us first show that $\DD^*\subset W^s(F_A)^*$. It is straightforward to check that, for any $m\geq 1$, the components of $A^{(n-l)(m+1)}$ satisfy:
\begin{equation*}
a^{[(n-l)(m+1)]}_{st}\geq (n-l)^{m-1}
\end{equation*}
for any $1\leq s\leq n$ and $l+1\leq t\leq n$. Thus, for such pairs $(s,t)$ one has $\lim_{m \to \infty} a^{[m]}_{st} = + \infty$. It follows that for any $z\in\DD^*$ one has $\lim_{m \to \infty} ||F_A^m(z)||=0$.

Now clearly $W^s(F_A)^*$ is invariant with respect to the action of $U_A$, so that $F_A^m(\DD^*)\subset W^s(F_A)^*$ for any $m\in\ZZ$. Finally, if $z\in W^s(F_A)^*$, then for big enough $m>0$ one has $F_A^m(z)\in\DD^*$.
\end{proof}

\begin{lem}\label{outD}
Let $A$ be a Kato matrix of type $l$ and $z\in\DD^*$. Then there exists $m\in\NN$ so that $F_A^{-m}(z)\notin F_A(\DD^*)$.
\end{lem}
\begin{proof}
After eventually replacing $F_A$ by $F_A^{n-l}$, we can suppose that $A$ is $l$-positive (see Remark~\ref{positive}). As in the proof of Lemma~\ref{descrW}, it follows inductively that the coefficients $a_{st}^{[m+1]}$ of $A^{m+1}$ satisfy $a_{st}^{[m+1]}\geq (n-l)^{m-1}$, for any $1\leq s\leq n$ and $ l+1\leq t\leq n$. Thus for any $z\in\DD^*$ one has:
\begin{equation}\label{Fm}
||F^{m+1}_A(z)||^2\leq n|z_{l+1}\cdots z_n|^{2(n-l)^{m-1}}.
\end{equation}

Suppose now on the contrary that for any $m\in\NN$, $z^{(m)}:=F_A^{-(m+1)}(z)\in F_A(\DD^*)$. Since $A$ is $l$-positive, by Lemma~\ref{contractie} there exists $\epsilon>0$ so that $F_A(\DD^*)\subset \DD_{1-\epsilon}$. Thus, there exists a subsequence $(z^{(m_j)})_{j\in\NN}$ that converges to $z^\infty\in\overline \DD_{1-\epsilon}$. In particular, one has 
\begin{equation}\label{limm}
\lim_{j\to\infty}|z_{l+1}^{(m_j)}\cdots z_n^{(m_j)}|=|z_{l+1}^\infty\cdots z_n^\infty|<1.
\end{equation}

On the other hand, by \eqref{Fm} one has:
\begin{equation*}
||z||^2=||F^{m_j+1}(z^{(m_j)})||^2\leq n|z_{l+1}^{(m_j)}\cdots z_n^{(m_j)}|^{2(n-l)^{m_j-1}}
\end{equation*}
so that
\begin{equation*}
1=\lim_{j\to\infty} (\frac{1}{n}||z||^2)^{(n-l)^{1-m_j}}\leq |z_{l+1}^\infty\cdots z_n^\infty|^2.
\end{equation*}
But this last equation contradicts \eqref{limm}, hence there exists some $m\in\NN$ for which $z^{(m)}\notin F_A(\DD^*)$. 
\end{proof}

\begin{prop}\label{compactification}
For any Kato matrix $A$, we have a biholomorphism:
\begin{equation*}
M_A-C\cong W^s(F_A)^*/U_A.
\end{equation*}
\end{prop}
\begin{proof}
Let $\tilde C\subset \tilde M_A$ be the lift of $C$ to $\tilde M_A$.  Clearly $\tilde C$ is fixed by $\pi_1(M_A)$, so that we have an action of $\pi_1(M_A)$ on $\tilde M_A-\tilde C$, whose quotient is precisely $M_A-C$. We thus need to find a biholomorphism between $\tilde M_A-\tilde C$ and $W^s(F_A)^*$ which is equivariant with respect to the actions of $\pi_1(M_A)$ and $U_A$ respectively. 

We suppose that $A$ is of type $l$. Moreover, since $\tilde M_{A^p}-\tilde C_{A^p}=\tilde M_A-\tilde C_A$ by Lemma~\ref{cov}, it suffices to prove the above for $A^p$, $p>0$, hence by Remark~\ref{positive}, we can suppose that $A$ is $l$-positive.

We start by defining $\psi: \tilde M_A-\tilde C\rightarrow \CC^l \times \Cn{n-l}$  by $\psi(x)=F^m_A\pi(x)$, where $m\in\ZZ$ is determined by the condition $x\in W_m$ and $\pi$ is the blow-down map from each copy $W_m$ of $W$ to $\DD$. Clearly, $\psi$ is holomorphic. Moreover, by Lemma~\ref{Hj}, $\pi^{-1}\{z_{l+1}\cdots z_n=0\}\cap W=C\cap W$, so that indeed $\psi$ takes values in $\CC^l \times \Cn{n-l}$. 

Note that $F_A: \CC^l \times \Cn{n-l} \rightarrow \CC^l \times \Cn{n-l}$ is an isomorphism. Indeed, writing $A=\begin{tiny}\begin{pmatrix}
  I_l & G \\
  0 & B
 \end{pmatrix}\end{tiny}$ with $L\in\ZZ^{n-l}$ denoting a line of $G$ and writing $F_A$ as in \eqref{standardform}:
\begin{equation*}
F_A(z,w)=(z_1w^L, \ldots,z_lw^L, w^B), \ \ z\in\CC^l, \ w\in(\CC^*)^{n-l}
\end{equation*}
it is easy to check that the holomorphic map
\begin{equation*} 
H_A(z,w)=(z_1w^{-LB^{-1}},\ldots,z_lw^{-LB^{-1}},w^{B^{-1}}), \ \ z\in\CC^l, \ w\in(\CC^*)^{n-l}
\end{equation*}
is the inverse of $F_A$.

Let us prove that $\psi$ is injective. Let $x \in W_m$ and $y \in W_q$ be such that $F^m_{A}(\pi(x))=F^q_{A}(\pi(y))$. Assume first that $m>q$, and let $x'$ and $y'$ be the copies of $x$ and $y$ respectively in $W$. We then have:
 \begin{equation*}
 F_A(\DD)\ni F_A^{m-q}(\pi(x'))=\pi(y')\in \pi(W).
 \end{equation*}
However, this is impossible since $\pi(W)=\DD-\pi(\sigma(\DD))=\DD-F_A(\DD)$. It follows then that $m=q$ and $x$ and $y$ stay in the same copy of $W$. Since $\pi$ is a biholomorphism outside the exceptional divisor and $F_A$ is also invertible, it follows that $x=y$.  
  
Thus, as $\psi$ is a holomorphic map between $n$-dimensional manifolds, it follows that it is a biholomorphism onto its image (see for instance \cite[Section~4]{shabat}). Let us show that $\psi(\tilde M_A-\tilde C)=W^s(F_A)^*$.

By definition, by Lemma~\ref{Hj} and by Lemma~\ref{descrW}, we have $\im \psi\subset \bigcup_{m\in\ZZ}F_A^m(\DD^*)=W^s(F_A)^*$. Let now $w\in W^s(F_A)^*$. In order to show that $w\in\im\psi$, it suffices to find $u\in\pi(W-C)=\DD^*-F_A(\DD^*)$ and $m\in \ZZ$ so that $w=F_A^m(u)$.  

By Lemma~\ref{descrW}, there exist $z\in\DD^*$ and $a\in\ZZ$ so that $w=F_A^a(z)$. If $z\notin F_A(\DD^*)$, then we simply take $m=a$ and $u=z$. If not, then by Lemma~\ref{outD}, there exists a minimal number $p>0$ so that $v:=F_A^{-p}(z)\notin F_A(\DD^*)$. By our choice of $p$, $F_A^{-p+1}(z)=F_A(v)=F_A(v_1)$ for some $v_1\in\DD^*$, but since $F_A$ is invertible, we find $v=v_1\in\DD^*$. Thus we can put $u=v$ and $m=a+p$.

Finally, the map $\psi$ is clearly equivariant with respect to the action of $\pi_1(M_A)$ on $\tilde M_A-\tilde C$ and the action on $U_A$ on $W^s(F_A)^*$. It follows that $\psi$ descends to a biholomorphism between $M-C$ and $W^s(F_A)^*/U_A$.
\end{proof}

\begin{lem}\label{stablerelation} If $A=\begin{tiny}\begin{pmatrix}
  I_l & * \\
  0 & B
 \end{pmatrix}\end{tiny}$ is a Kato matrix of type $l$, then $W^s(F_A)^*=\CC^l \times W^s(F_B)^*$.
\end{lem}
\begin{proof}
Using \eqref{standardform}, we write:
\begin{equation*}
F_{A^m}(z)=(z_1G^{(m)}(z_{l+1}, \ldots, z_n), \ldots,z_lG^{(m)}(z_{l+1}, \ldots, z_n), F_{B^m}(z_{l+1}, \ldots, z_n)),
\end{equation*}
hence the inclusion $\subseteq$ is clear. Conversely, let $z=(z_1, \ldots, z_n) \in \CC^l \times W^s(F_B)^*$. Since we have $\lim_{m \to \infty} F^m_{B}(z_{l+1}, \ldots, z_n)=0$, there exists $m_0>0$ such that: 
\begin{equation*}
(z'_{1}, \ldots, z'_n):=F_{A^{m_0}}(z) \in \CC^l\times\DD\subset\CC^{n}
\end{equation*} 
where $\BB$ is the ball in $\CC^{n-l}$. Therefore:  
\begin{align*}
F_{A^{m+m_0}}(z_1, \ldots, z_n)&=F_{A^m}(z'_1, \ldots, z'_n)\\
&=(z'_1G^{(m)}(z'_{l+1}, \ldots, z'_n), \ldots,z'_lG^{(m)}(z'_{l+1}, \ldots, z'_n), F_{B}^m(z'_{l+1}, \ldots, z'_n))
\end{align*}
tends to 0 as $m \to \infty$, since for any $1\leq s\leq n$ and  $l+1 \leq t \leq n$, $|z'_{t}|<1$ and $\lim_{m \to \infty} a^{[m]}_{st}=\infty$. 
\end{proof}


\begin{thm}\label{Thmcompact} Let $A=\begin{tiny}\begin{pmatrix}
  I_l & G \\
  0 & B
 \end{pmatrix}\end{tiny}$ be a Kato matrix. Then we have a biholomorphism:
\begin{equation*}
M_A-C\cong\HH\times\CC^{n-1}/U_A\ltimes \Lambda
\end{equation*}
where $\Lambda\subset\CC^{n}$ is a rank $n-l$ lattice and $U_A=\langle F_A \rangle \cong\ZZ$. Moreover, if $l>0$, then there exists a finite ramified covering of degree $n-l-1$:
\begin{equation*}
\Phi: M_A-C_A \rightarrow \CC^l\times (M_B-C_B)
 \end{equation*}
 where $C_A$ and $C_B$ denote the corresponding cycles of divisors in the Kato manifolds $M_A$ and $M_B$ respectively. In particular, for $l=n-2$, $\Phi$ is a biholomorphism.
\end{thm}

\begin{proof}
Let us first suppose that $l=0$. Using Proposition~\ref{compactification}, the point is to identify $W^s(F_A)^*\subset\Cn{n}$ with a quotient of $\HH\times\CC^{n-1}$ by $\Lambda\cong \ZZ^n$. Since $W^s(F_A)^*=W^s(F_{A^p})^*$ for any $p>0$, we can suppose without loss of generality that $A$ is positive (see Remark~\ref{positive}).

Let us write $V=\CC^n$ and let us denote by $e_1^*,\ldots, e_n^*$ the standard dual basis of $V^*$. Note that $\Lambda=i\ZZ^n\subset V$ acts by translations on $V$, so that  $\Cn{n}=V/\Lambda$. The natural projection is given by: 
\begin{equation*}
p:\CC^n\rightarrow \Cn{n}, \ \ z=(z_1,\ldots, z_n)\mapsto (\e^{2\pi z_1},\ldots, \e^{2\pi z_n}).
\end{equation*}
Moreover, a Kato matrix $A\in\GL(n,\ZZ)$ acts on $V$ linearly by:
\begin{equation*} 
Az=(\sum_{j=1}^na_{1j}z_j,\ldots, \sum_{j=1}^na_{nj}z_j)
\end{equation*} 
while fixing $\Lambda$, and $A$ also acts on $\Cn{n}$ via $A.z=F_A(z)$. In this way, the map $p$ is equivariant with respect to the two actions of $U_A:=<A>\cong \ZZ$.

Since $A$ has only positive components, the  Perron-Frobenius theorem (see for intance \cite[Chapter~8]{m00}) implies that $A$ has a simple real eigenvalue $\al>0$ such that for any other eigenvalue $\be\in\Spec(A)$, we have $|\be|<\al$. Moreover, the eigenspace of $\al$ contains a vector $f\in V$, called a Perron vector, with $e_j^*(f)>0$ for all $1\leq j\leq n$. 

For $\be\in\Spec(A)$, let $V(\be)\subset V$ denote the generalized eigenspace of $\be$: 
\begin{equation*}
V(\be)=\{v\in\CC^n| \exists p>0 \ (A-\be \I_n)^pv=0\}
\end{equation*}
and define 
\begin{equation*}
V_0=\bigoplus_{\substack{\be\in \Spec(A)\\ \be\neq \al}}V(\be)  
\end{equation*}
so that we have $V\cong V(\al)\oplus V_0$. 

Choose a Perron vector for $A^t$ acting on $V^*$,  $f^*\in V(\al)^*$, so that $\langle f^*,f\rangle=1$. Write $a\in\Aut(V)$ for the automorphism defined by $v\mapsto Av$. Then $a(V_0)\subset V_0$, so $a|_{V_0}$ induces $a_0\in\Aut(V_0)$ and we have $a=\al f^*\otimes f+a_0$. 

\begin{lem}\label{limL}
For any $u\in V_0$, we have $\lim_{m\to\infty}\frac{1}{\al^m}a_0^mu=0$.
\end{lem}
\begin{proof}
There exists a basis of $V_0$ with respect to which $a_0$ is represented by a matrix in Jordan normal form. Let $\be\neq \al$ be an eigenvalue of $A$ and let $f_1,\ldots, f_r$ be a basis of generalized eigenvectors corresponding to a Jordan block for $\be$. Namely, they verify:
\begin{equation*}
a_0f_1=\be f_1, a_0f_j=\be f_j+f_{j-1} \text { for } 2\leq j \leq r. 
\end{equation*}
Since one has $|\frac{\be}{\al}|<1$, we obtain $\lim_{m \to\infty}\frac{1}{\al^m}a_0^m f_1=\lim_{m\to\infty}\frac{\be^m}{\al^m} f_1=0$. For $j\geq 2$, it can be shown inductively on $m\geq 1$ that:
\begin{equation*}
a_0^m f_j=\sum_{p=0}^{\min{(j-1,m)}}{m\choose p}\be^{m-p}f_{j-p}.
\end{equation*} 
But:
\begin{equation*}
\lim_{m\to\infty}\frac{\be^{m-p}}{\al^m}{m\choose p}=0, \ \ \text{ for any }p\geq 0
\end{equation*} 
from which it follows that $\lim_{m\to\infty}\frac{1}{\al^m}a_0^mf_j=0$ for any $1\leq j\leq r$. Finally, $V_0$ is spanned by such bases of generalized eigenvectors, hence the conclusion follows. \end{proof}

Let now $\Omega:=p^{-1}(W_s(F_A)^*)\subset V$.  Write $\re:V=\RR^n\oplus i\RR^n\rightarrow\RR^n$ for the natural projection, consisting in taking the real parts of the standard coordinates. We then have the following equivalences:
\begin{align}\label{omegaa}
\nonumber v\in\Omega &\Leftrightarrow \lim_{m\to \infty} |F_A^m(p(v))| = 0 \Leftrightarrow \lim_{m\to \infty} |p(\re A^mv )|= 0\\
&\Leftrightarrow \lim_{m\to \infty} e_j^*(\re A^m v)= -\infty, \ \ \ \forall j\in\{1,\ldots, n\}.
\end{align}

Let $v\in\Omega$ and write $v=\lambda f+v_0$, with $\la\in\CC$, $v_0\in V_0$. We then have:
\begin{equation*}
-\infty=\lim_{m\to\infty} e_j^*(\re A^mv)=\lim_{m\to\infty}\al^m(\re(\la)e_j^*(f)+e_j^*(\re(\frac{1}{\al^m}a_0^mv_0))), \ \ 1\leq j\leq n.
\end{equation*} 
Hence, Lemma~\ref{limL} implies that $\re(\la)\leq 0$. As $\Omega$ is open, it follows therefore that $\Omega\subset i\HH f\oplus V_0$. Conversely, Lemma~\ref{limL} implies that any $v\in i\HH f\oplus V_0$ satisfies condition \eqref{omegaa} and we find $\Omega=i\HH f\oplus V_0$.

Finally, clearly $\Omega$ is preserved by $\Lambda$, so $W^s(F_A)^*=i\HH f\oplus V_0/\Lambda$. Note that indeed $U_A$ acts on $\Lambda$ linearly, which gives the semi-direct product structure of $U_A\ltimes \Lambda$. Applying Proposition~\ref{compactification}, the conclusion follows.

Suppose now $l>0$. Then, by Lemma~\ref{stablerelation}, we have $W^s(F_A)^*=\CC^l\times W^s(F_B)^*$, while from the first part of this proof, $W^s(F_B)^*\cong \HH\times \CC^{n-l-1}/\Lambda_0$, where $\Lambda_0\subset\CC^{n-l}$ is a rank $n-l$ lattice. Putting $\Lambda:=0\oplus\Lambda_0\subset\CC^n=\CC^l\oplus\CC^{n-l}$ and applying Proposition~\ref{compactification}, the first part of the Theorem follows.

For the second part, let us denote by $z=(z_1,\ldots, z_l)$ the holomorphic coordinates on $\CC^l$ and by $w=(w_{l+1},\ldots, w_n)$ the holomorphic coordinates on $W^s(F_B)^*\subset(\CC^*)^{n-l}$, so that we have, via \eqref{standardform}, $F_A(z,w)=(z_1w^L,\ldots, z_lw^L, F_B(w))$, where $L\in\NN^{n-l}$ denotes a line of $G$. 

Let $m:=n-l-1\geq 1$ and $J_0=(1,\ldots,1)\in\ZZ^{n-l}$, and recall that by   Lemma~\ref{vectprJ} we have $J_0B-J_0=mL$. Let us define the holomorphic map:
\begin{equation*}
\hat\Phi:W^s(F_A)^*\rightarrow\CC^l\times W^s(F_B)^*, \ \ \hat\Phi(z,w)=(z_1^mw^{-J_0},\ldots, z_l^mw^{-J_0},w).
\end{equation*}
It is clearly a branched covering map of degree $m$. Moreover, using the properties of $J_0$, we find that $\hat\Phi\circ F_A=(\id_{\CC^l}\times F_B)\circ\hat\Phi$. We infer thus, via Proposition~\ref{compactification}, that $\hat\Phi$ descends to a degree $m$ map to the quotient:
\begin{equation*}
\Phi:M_A-C_A=W^s(F_A)^*/U_A\rightarrow \CC^l\times W^s(F_B)^*/U_B=\CC^l\times (M_B-C_B).
\end{equation*}
Clearly, in the case $l=n-2$, $\hat\Phi$ has a holomorphic inverse given by $\hat\Phi^{-1}(z,w)=(z_1w^{J_0},\ldots, z_lw^{J_0},w)$, which concludes the proof.
 \end{proof}


\begin{cor} If $A$ is a Kato matrix of type $l$, the fundamental group of $\tilde{M}_A - \tilde{C}$ is isomorphic to $\mathbb{Z}^{n-l}$.
\end{cor} 

\begin{cor} If $A$ is a Kato matrix of type $l$, the fundamental group of $M_A - C$ is isomorphic to $\mathbb{Z} \ltimes \mathbb{Z}^{n-l}$. Moreover, if $\iota:M_A-C\rightarrow M_A$ denotes the inclusion and $\iota_*$ the map induced on fundamental groups, then we have the following commutative diagram:
\begin{equation}\label{diagrcom}
\xymatrix{
0\ar[r] &\Lambda \ar[r]\ar[d]^0 &\pi_1(M_A-C) \ar[r]\ar[d]^{\iota_*} & U_A\ar[r]\ar[d]^j &0\\
0 \ar[r] &0 \ar[r] &\pi_1(M_A)\ar[r]^k &\ZZ \ar[r] &0
}
\end{equation}
where $k$ denotes the canonical isomorphism $\pi_1(M_A)\cong\ZZ$ and $j$ denotes the isomorphism $U_A=\{F_A^m, m\in\ZZ\}\rightarrow \ZZ$, $F_A^m\mapsto m$.
\end{cor}
\begin{proof}
By the definition of the isomorphism $\Phi:W^s(F_A)^*\rightarrow\tilde M_A-\tilde C$ in the proof of Proposition~\ref{compactification}, we have the commutative diagram of covering maps:
\begin{equation*}
\xymatrix{
&W^s(F_A)^*\ar[r]^{\tilde\iota\circ\Phi} \ar[d]^{U_A} &\tilde M \ar[d]^{\pi_1(M)}\\
&M_A-C \ar[r]^{\iota} &M 
}
\end{equation*}
where $\tilde\iota:\tilde M-\tilde C\rightarrow \tilde M$ denotes the inclusion. This implies that the second square in \eqref{diagrcom} is commutative. The first square in \eqref{diagrcom} is induced by $\tilde\iota\circ\Phi$ and clearly $0=(\tilde\iota\circ\Phi)_*:\Lambda=\pi_1(W^s(F_A)^*)\rightarrow \pi_1(\tilde M)=0$. 
\end{proof}

\begin{cor}
Any lcK metric on $M_A$ induces a strict lcK metric on $M_A-C$.
\end{cor}
\begin{proof}
By the above Corollary, $H^1(M_A-C,\RR)\cong\RR$ and $i^*:H^1(M_A,\RR)\rightarrow H^1(M_A-C)$ is an isomorphism, hence the Lee form of any induced lcK metric on $M_A-C$ is not exact.
\end{proof}

\section{Holomorphic vector fields and forms}\label{secHol}

\begin{prop}\label{vf}
Let $A\in\GL(n,\ZZ)$ be a Kato matrix of type $l$, and let $m(1)\geq l$ denote the geometric multiplicity of $1$ as an eigenvalue of $A$. If $l=0$ and $A$ is positive, then we have:  
\begin{equation*}
\dim_{\CC}H^0(M_A, TM_A)=m(1). 
\end{equation*}
For general $A$, we have the inequality:
\begin{equation}\label{ineqV}
\dim_{\CC}H^0(M_A, TM_A)\geq l^2-l+m(1).
\end{equation} 
\end{prop}
\begin{proof}
Recall that if $\pi:\hat \DD\rightarrow \DD$ is the blow-up of $\DD\subset\CC^n$ at $0$ and $E=\pi^{-1}(0)\cong\CP^{n-1}$ is the exceptional divisor, then one has an exact sequence of vector bundles over $E$:
\begin{equation}\label{exE}
\xymatrix{
0\ar[r]& TE\ar[r] &T\hat \DD|_E\ar[r] &\mathcal{O}_E(-1)\ar[r] &0
}
\end{equation}
where $\mathcal{O}_E(-1)\cong \mathcal{O}_{\CP^{n-1}}(-1)$ is precisely the restriction to $E$ of the line bundle on $\hat \DD$ determined by $E$ as a divisor. In particular, if $Z$ is a holomorphic vector field on $\hat \DD$, then as $H^0(E,\mathcal{O}_E(-1))=0$, it follows that $Z|_E$ is tangent to $E$. 

Let us suppose first that $l=0$ and $A=(a_{kj})_{k,j}$ is positive. Take $Z\in H^0(M_A, TM_A)$. Then we have $Z|_W\in H^0(W,TW)$ which, by Hartogs' theorem, we can extend to $\hat Z\in H^0(\hat\DD^{(k)}, T\hat\DD^{(k)})$. 

Consider now $X=(\sigma^{-1})_*\hat Z|_{\sigma(\DD)}\in H^0(\DD, T\DD)$.   By Lemma~\ref{Hj}, $\sigma^{-1}(E)=\{z_{1}\cdots z_n=0\}$, so by the above considerations $X$ must be tangent to the hypersurfaces $\{z_j=0\}$, $j\in\{1,\ldots n\}$. This easily implies that $X$ is of the form:
\begin{equation}\label{Wtg}
X_z=\sum_{j=1}^n z_jh_j(z)\frac{\del}{\del z_j}, \ \ z\in\DD
\end{equation} 
where $h_1,\ldots h_n$ are holomorphic functions on $\DD$. 

Moreover, as $\hat Z$ verifies $\gamma_*\hat Z=\hat Z$, it follows that $(F_A)_*X=X$, which also reads:
\begin{equation}\label{sysV}
\sum_{j=1}^na_{sj}h_j(z)=h_s(F_A(z)), \ \ s\in\{1,\ldots, n\}.
\end{equation}

Let us write the functions $h_j$, for $1\leq j\leq n$, as a power series: $h_j(z)=\sum_{I}c^{(j)}_Iz^I$, where $I=(i_1,\ldots, i_n)$ runs over $\NN^n$ and $z^I=z_1^{i_1}\cdots z_n^{i_n}$. Then \eqref{sysV} reads:
\begin{equation*}
\sum_{j=1}^n\sum_{I\in\NN^n}a_{sj}c^{(j)}_{I}z^I=\sum_{I\in\NN^n}c^{(s)}_{I}z^{IA}, \ \ s\in\{1,\ldots, n\}.
\end{equation*}
In degree $0$, this equation implies that $c=(c^{(1)}_{(0,\ldots, 0)},\ldots, c^{(n)}_{(0,\ldots, 0)})^t$ is an eigenvector of $A$ with eigenvalue $1$. Moreover, for any $I>0$ and $s>0$, we have $c_I^{(s)}=0$. Indeed, if this were not the case, then take $I_0>0$ minimal (with respect to the lexicographic order), such that there exists $j\in\{1,\ldots, n\}$ with $c_{I_0}^{(j)}\neq 0$. But then we obtain:
\begin{equation*}
\sum_{j=1}^na_{sj}c^{(j)}_{I_0}=c_{I_0A^{-1}}^{(s)}=0 
\end{equation*}    
by the minimality of $I_0$, as clearly $I_0A^{-1}<I_0$. On the other hand this is impossible, since the matrix $A$ is invertible and the vector $c=(c^{(1)}_{I_0},\ldots, c^{(n)}_{I_0})^t$ is not zero. Thus we have shown that $\dim_\CC H^0(M_A, TM_A)\leq m(1)$.

Let us now suppose that $A$ is any Kato matrix $A=\begin{tiny}\begin{pmatrix}
  I_l & G \\
  0 & B
 \end{pmatrix}\end{tiny}$ and write, via Lemma~\ref{standardform},
 \begin{equation*} 
 F_A(z,w)=(z_1w^L,\ldots, z_lw^L, F_B(w)), \ \ (z,w)\in\CC^{l}\times \CC^{n-l}
 \end{equation*} 
 where $L\in\NN^{n-l}$. Note that a vector $x=(u,v)\in\CC^l\times \CC^{n-l}$ verifies $Ax^t=x^t$ if and only if $Gv^t=0$ and $Bv^t=v^t$. It follows that $Q:=\{v\in \CC^{n-l}, Bv^t=v^t, Gv^t=0\}$ is of complex dimension $m(1)-l$. 
 
 Now, for any $v=(v_{l+1},\ldots, v_n)\in Q$, and for any $d_{s,t}\in\CC$, $1\leq s,t\leq l$, it is straightforward to check that the vector field on $\CC^n$:
 \begin{equation}
 X=\sum_{s,t=1}^ld_{st}z_s\frac{\del}{\del z_t}+\sum_{p=l+1}^n v_pw_p\frac{\del }{\del w_p}
 \end{equation}
verifies $(F_A)_* X=X$. Since it moreover vanishes in $0$, it lifts, via $\pi$, to a holomorphic vector field on $W$ which is invariant to the gluing map, and hence it defines a holomorphic vector field on $M_A$.  We have thus shown the inequality \eqref{ineqV}.
\end{proof}

\begin{cor}\label{torus}
Let $A\in\GL(n,\ZZ)$ be a Kato matrix of type $l>0$. Then the compact torus $\mathbb{T}^l$ acts effectively by biholomorphisms on $M_A$.
\end{cor}
\begin{proof}
By the above proof, the real holomorphic vector fields on $\CC^n$ given by
\begin{equation*} 
X_1=\re(iz_1\frac{\del}{\del z_1}),\ldots, X_l=\re(iz_l\frac{\del}{\del z_l})
\end{equation*} 
induce $l$ linearly independent real holomorphic vector fields $Z_1,\ldots, Z_l$ on $M_A$. For $1\leq j\leq l$, denote by $\Phi^t_j(z)=(z_1,\ldots,z_{j-1},\e^{it}z_j,z_{j+1}\ldots, z_n)$ and by $\Psi^t_j$ the one-parameter groups of $X_j$ and of $Z_j$ respectively.  

Fix $1\leq j\leq l$. Note that since $\Phi_j^t$ commutes with $F_A$, this implies that $\Psi_j^t$ preserves $W\subset M_A$. Moreover, since $\pi_* Z_j|_W=X_j$, we have $\pi\circ\Psi^t_j|_W=\Phi^t_j\circ\pi|_W$. Therefore, as $\Phi_j^{2\pi}=\id$ on $\CC^n$, we infer that $\Psi_j^{2\pi}|_{W-C}=\id$. However $W-C$ is a dense subset of $M_A$, hence $\Psi_j^{2\pi}=\id$ on $M_A$. 

Finally, since $[X_j,X_m]=0$ it follows that $[Z_j,Z_m]=0$ for any $1\leq j,m\leq l$.  Thus indeed the vector fields $Z_1,\ldots, Z_l$ generate the effective action of a compact torus $\mathbb{T}^l$ on $M_A$.
\end{proof}

\begin{prop}\label{forme}
Let $A\in\GL(n,\ZZ)$ be a Kato matrix. Then $H^0(M_A,\Omega^1)=0$.
\end{prop}
\begin{proof}
Since $M_{A^p}\rightarrow M_A$ is a finite covering for $p>0$, it follows that $\dim H^0(M_A,\Omega^1)\leq \dim H^0(M_{A^p},\Omega^1)$. Thus it suffices to prove the assertion for $A^p$, so we can suppose, via Remark~\ref{positive}, that $A$ is $l$-positive.

Consider, as in the beginning of the proof of Proposition~\ref{vf}, the dual of the exact sequence \eqref{exE}. Then, as $H^0(E,\Omega^1)=0$, it follows that for any holomorphic form $\al\in H^0(\hat \DD,\Omega^1)$, $\al|_E$ vanishes when evaluated on vector fields tangent to $E$.

Let now $\al\in H^0(M_A, \Omega^1)$. In the same way as before, $\al$ determines $\hat \al\in H^0(\hat \DD^{(k)},\Omega^1)$ and then $\be=\sigma^*\hat\al\in H^0(\DD,\Omega^1)$. Moreover, by the above considerations and by Remark~\ref{Hj}, we find that $\be$ is of the form
\begin{equation*}
\be=\sum_{j=1}^l h_j(z,w)dz_j+\sum_{s=l+1}^n \frac{w_l\cdots w_n}{w_s}h_s(z,w)dw_s
\end{equation*}
where $h_1, \ldots, h_n$ are holomorphic functions on $\DD$. 
Now, writing $F_A(z,w)=(z_1w^L,\ldots, z_lw^L,F_B(w))$ with $L=(p_{l+1},\ldots, p_n)\in\NN^{n-l}$, the condition $F_A^*\be=\be$ translates into:
\begin{align}
\label{eq5} F_A^*h_jw^L&=h_j, &1\leq j\leq l\\ 
\label{eq6} p_t\sum_{j=1}^lF_A^*h_jz_jw^L&+\sum_{s=l+1}^na_{st}F_{l+1}\cdots F_n F_A^*h_s=w_{l+1}\cdots w_n h_t, &l+1\leq t\leq n
\end{align}
where $F_{l+1},\ldots, F_n$ denote the last $n-l$ components of $F_A$. 

Now we proceed as in Proposition~\ref{vf} and write $h_1,\ldots,h_n$ as power series in $z$ and $w$. Then we find, as before, that $h_1=\ldots=h_l=0$ from \eqref{eq5}, and then also that $h_{l+1}=\ldots=h_n=0$ from \eqref{eq6}. Thus $\be=0$ and the conclusion follows. 
\end{proof}

\section{Analytical invariants}\label{secInv}

\begin{prop}\label{algdim}
Let $A\in\GL(n,\ZZ)$ be a Kato matrix and let $m(1)$ denote the geometric multiplicity of $1$ as an eigenvalue of $A$. Then the algebraic dimension of $M_A$ satisfies $n>a(M_A)\geq m(1)$. In particular, if $A$ is of type $l$ with $l>0$, then $a(M_A)\geq l$. Moreover, if $l=n-2$, then $a(M_A)=n-2$. 
\end{prop}
\begin{proof}
Consider the map $L:\ZZ^n\rightarrow \ZZ^n$, $I\mapsto IA-I$, and let $Q=\ker L\subset\ZZ^n$. Then $Q$ is a lattice of rank $p:=m(1)$, by hypothesis, and we fix a basis $I_1,\ldots, I_p$ of $Q$.  

Consider now the meromorphic functions $f_j:\CC^n\dasharrow\CC$, $f_j(z)=z^{I_j}$, for $1\leq j\leq p$. Since $I_jA=I_j$, they are invariant under the action of $F_A$:
\begin{equation*}
(F_A^*f_j)(z)=f_j(z^A)=z^{I_jA}=z^{I_j}=f_j(z).
\end{equation*}
Thus, the functions $\pi^*f_j$ on $W$ are invariant to the gluing map $\sigma\circ\pi$ and define meromorphic functions $\hat f_j$ on $M_A$, for $1\leq j\leq p$. 

Let us check that they are algebraically independent. It is enough to check that this happens on $W-E\subset M_A$, which is biholomorphic to $\pi^{-1}(W-E)$ via $\pi$. Therefore it is enough to check that the functions $f_1,\ldots, f_p$ are algebraically independent on $\D-F_A(\D)\cap(\CC^*)^n\subset\pi^{-1}(W-E)$. Suppose thus that there exists $P\in\CC[X_1,\ldots, X_p]$ with $P(f_1,\ldots, f_p)=0$. If we write $P=\sum_K a_KX^K$, where $K=(k_1,\ldots, k_p)$ runs over a finite set of $\mathbb{N}^p$ and $X^K=X_1^{k_1}\cdots X_p^{k_p}$, we find:
\begin{align*}
0=P(f_1,\ldots, f_p)=\sum_{K=(k_1,\ldots, k_p)} a_K z^{k_1I_1}\cdots z^{k_pI_p}=\sum_{K=(k_1,\ldots, k_p)}a_Kz^{k_1I_1+\ldots+k_pI_p}.
\end{align*} 
But now, since the functions $z_1,\ldots, z_n$ are algebraically independent, it follows that for each $K=(k_1,\ldots, k_p)$ with $a_K\neq 0$ we have $k_1I_1+\ldots+k_pI_p=0$. However, since $I_1,\ldots, I_p$ are linearly independent, this implies then that $K=0\in\NN^p$ and $P(f_1,\ldots, f_p)=a_0=0$. Thus $P=0$ and the conclusion follows.

On the other hand, we cannot have $a(M_A)=n$, since then $M_A$ would satisfy the $\del\overline\del$-lemma and would admit a Hodge decomposition. However, this is impossible since $b_1(M_A)=1$. Note that when $A$ is of type $l>0$ we have $m(1)\geq l>0$ by Remark~\ref{eigenv1}. 

In the case $l=n-2$, we have $\begin{tiny}A=\begin{pmatrix}
  I_{n-2} & G \\
  0 & B
 \end{pmatrix}\end{tiny}$.  We can apply \cite[Theorem 3.8, (ii)]{ueno}, according to which if $M$ is a complex submanifold of $N$, then $a(N) \leq a(M)+codim(M)$. Applying this inequality for $N=M_{A}$ and $M=M_B$, we get $n-2\leq a(M_A) \leq a(M_B)+n-2$. However, $M_B$ is a Kato surface and it is well known that is has algebraic dimension 0, therefore $a(M_A)=n-2$.
\end{proof}

\begin{rmk}
Let $A$ be Kato matrix of type $l$. By Lemma~\ref{vectprJ}, the vectors $I_j=((n-l-1)e_j,-J_0)\in\ZZ^n$, $1\leq j\leq l$ are $l$ linearly independent vectors satisfying $I_jA=A$, where $e_1,\ldots, e_l$ is the standard basis of $\RR^l$ and $J_0=(1,\ldots, 1)\in\NN^{n-l}$. It follows that we have $l$ algebraically independent meromorphic functions $\hat f_1,\ldots, \hat f_l$ on $M_A$ induced by the $F_A$-invariant functions on $\CC^n$: 
\begin{equation*}
f_j(z)=\frac{z_j^{n-l-1}}{z_{l+1}\cdots z_n}, \ \ \ 1\leq j\leq l.
\end{equation*}
\end{rmk}

\begin{rmk}\label{linEquiv}
Let $A$ be a Kato matrix of type $l$. For $1\leq j\leq l$, let $A_{(j)}\in\GL(n-1,\ZZ)$ be the Kato matrix obtained by erasing the $j$-th line and column from $A$. Pulling back via $\pi$ the divisor $\{z_j=0\}\subset\CC^n$ and then gluing via $\gamma$, we obtain precisely $M_{A_{(j)}}\subset M_A$, by Remark~\ref{subvarKato}. On the other hand, the divisor on $M_A$ induced by the same procedure from $\{z_{l+1}\cdots z_n=0\}\subset\CC^n$  gives $C_A$, by Lemma~\ref{Hj}. Hence we find: 
\begin{equation*}
div(\hat f_j)=(n-l-1)M_{A_{(j)}}-C_A, \ \ 1\leq j\leq l
\end{equation*}
i.e. the functions $\hat f_j$ give linear equivalences between the divisors $(n-l-1)M_{A_{(j)}}$ and $C_A$. 
\end{rmk}

\begin{prop}\label{Kodairadimension} Let $A$ be a Kato matrix of type $l$ and let $\mathcal L$ be the flat line bundle on $M_A$ associated to $\rho:\ZZ=\pi_1(M_A)\rightarrow \ZZ/2\ZZ$, $\rho(m)=(-1)^m$. If $\det A=1$, then we have:
\begin{equation*}
K_{M_A}=\mathcal O(-A_{(1)}-\ldots A_{(l)}-C_A) \ \text{ hence } \ K_{M_A}^{n-l-1}=\mathcal O(-(n-1)C_A).
\end{equation*}
Otherwise, we have:
\begin{equation*}
K_{M_A}=\mathcal L\otimes\mathcal O(-A_{(1)}-\ldots A_{(l)}-C_A) \ \text{ hence }\  K_{M_A}^{n-l-1}=\mathcal L^{n-l-1}\otimes\mathcal O(-(n-1)C_A).
\end{equation*}
In particular, the Kodaira dimension of $M_{A}$ is $-\infty$.
\end{prop}

\begin{proof} 
Let us first assume that $\det A=1$. Consider the meromorphic $n$-form on $\CC^n$:
\begin{equation*}
\Omega=\frac{dz_1 \wedge \ldots \wedge dz_n}{z_1\cdots  z_n}.
\end{equation*}
Since $F_A^*\Omega=\Omega$, it follows that this form induces a meromorphic $n$-form $\omega$ on $M_A$. Moreover, by Remark~\ref{linEquiv}, we have:
\begin{equation*}
-K_{M_A}=-div(\omega)=A_{(1)}+\ldots A_{(l)}+C_A
\end{equation*}
from which the conclusion follows, using again Remark~\ref{linEquiv}. 

Otherwhise, we have $\det A=-1$, but then $F_A^*\Omega=-\Omega$. Hence $\Omega$ induces an $\mathcal L$-valued meromorphic $n$-form $\omega$ on $M_A$. The desired result follows similarly, using that $\mathcal L^*=\mathcal L$.  
\end{proof} 

\begin{rmk}
When $\det A=1$, the above result gives $K_{M_A}=\mathcal O(-C_A)$ in the case $l=0$ and $K_{M_A}=\mathcal O(-(n-1)C_A)$ in the case $l=n-2$.
\end{rmk}

\begin{prop}\label{holofunct}
For a positive Kato matrix $A$ satisfying that $Spec(A)$ does not contain any root of unity, there are no non-constant holomorphic functions on $M_A-C$. 
\end{prop}
\begin{proof}
Suppose $f$ is a holomorphic function on $M_A-C$. By Proposition~\ref{compactification}, $f$ induces then a holomorphic function $g$ on $T:=W^s(F_A)^*$ which is $U_A=<F_A>$-invariant. 

Now $T\subset\CC^n$ is clearly a Reinhardt domain centered in $0$, meaning that it is closed under rotations in each coordinate. Moreover it is relatively complete as it does not intersect any of the hyperplanes $\{z_j=0\}$. Thus, $g$ can be expressed as a Laurent series:
\begin{equation}\label{ser}
g(t)=\sum_{I\in \ZZ^n}a_It^I
\end{equation}
which converges absolutely and uniformly on every compact subset of $T$ (see for instance \cite[Chapter~1]{shabat}). Here, for $I=(j_1,\ldots, j_n)\in\ZZ^n$ and $t=(t_1,\ldots, t_n)\in T$, we denote by $t^I=t_1^{j_1}\cdots t_n^{j_n}$.

Moreover, as $g$ is invariant under the action of $F_A$, we have:
\begin{equation*}
F_A^*g=\sum_{I\in\ZZ^n}a_It^{IA}=\sum_{I\in \ZZ^n} a_I t^I
\end{equation*}
hence $a_I=a_{IA}$ for any $I\in\ZZ^n$.  

Let us denote by $Q=\ZZ^n-\{0\}/\langle A\rangle$ the orbit set of the action of the group generated by $A$ on $\ZZ^n-\{0\}$. Since $1\notin\Spec(A^m)$ for any $m>1$, we have a partition into infinite sets $\ZZ^n-\{0\}=\sqcup_{I\in Q}\{IA^m, m\in\ZZ\}$. Then, by  the absolute convergence of \eqref{ser} and by Fubini, we have:
\begin{equation*}
g(t)=a_0+\sum_{I\in Q}a_I\sum_{m\in\ZZ} t^{IA^m}
\end{equation*}
and for each $I\in Q$ such that $a_I\neq 0$, the series $u_I(t)=\sum_{m\in \ZZ}t^{IA^m}$ converges absolutely. 

Let us fix $I=(j_1,\ldots, j_n)\in Q$ with $a_I\neq 0$. Writing $t=p(z)=(\e^{2\pi z_1},\ldots, \e^{2\pi z_n})$ for $z\in\Omega=p^{-1}(T)$, we have:
\begin{equation*}
t^{IA^m}=F_A^m(p(z))^I=p(A^mz)^I=\e^{2\pi \langle I^*, A^mz\rangle}
\end{equation*}
which must tend to $0$ in absolute value when $m$ goes to $\pm \infty$. Here we denoted by $I^*=\sum_{k=1}^nj_ke_k^*\in(\ZZ^n)^*\subset (\CC^n)^*$ and $\langle\cdot, \cdot \rangle$ stands for the dual pairing. In particular, it follows that the function $v_m(z)=\langle I^*, \re A^{-m} z\rangle$ on $\Omega$ tends to $-\infty$ when $m$ goes to $\infty$. However, if we choose a Perron vector $f\in(\RR_{>0})^n$ for $A$ with corresponding eigenvalue $\al>1$, then $f\in \Omega$ and we have:
\begin{equation*}
\lim_{m\to\infty}v_m(f)=\lim_{m\to\infty}\langle I^*, \al^{-m}f\rangle=0
\end{equation*}
which gives a contradiction. It follows thus that for any $I\in Q$, $a_I=0$, hence $g$ and also $f$ are constant.
\end{proof}

\section{Kato manifolds with $l=n-2$}\label{secn-2}

In this section we consider Kato manifolds $M_A$ associated to a Kato matrix $A\in\GL(n,\ZZ)$ of type $l=n-2$, where $n\geq 3$. For these manifolds, we are able to make precise all the invariants we have considered so far.

In this case, $A$ is of the form $\begin{tiny}\begin{pmatrix}
  I_{n-2} & G \\
  0 & B
 \end{pmatrix}\end{tiny}$ and $B\in\GL(2,\ZZ)$. In addition, it is easy to see that $B$ does not admit $1$ as an eigenvalue. For instance, one can apply the Perron-Frobenius theorem to a power $B^p$ which is a positive matrix, which together with $\det B=\pm 1$ implies that $B^p$ has two real eigenvalues $\al>1$ and $\pm \frac{1}{\al}$. It thus follows by Lemma~\ref{vectprJ} that we have $m_A(1)=n-2$, where $m_A(1)$ is the geometric multiplicity of $1$ as an eigenvalue of $A$. Moreover, the vector $J_0=(1,1)\in\ZZ^2$ satisfies $J_0B-J_0=L$, where $L$ is a line of $G$.
 
 Following \cite[Chapter 1, Section 3]{ueno} we recall that \textit{the algebraic reduction} of a  connected compact complex manifold $M$ is a surjective holomorphic map $r: M^*_A \rightarrow V$, where $M^*_A$ is a smooth manifold which is bimeromorphic to $M_A$, $V$ is a smooth projective manifold of dimension $a(M_A)$ and $r$ induces an isomorphism of the fields of meromorphic functions. Moreover, this is unique up to bimeromorphic equivalences, and the fibers of $r$ are connected.
 
 Let us recall that in our situation, by the proof of Proposition~\ref{algdim}, we have $n-2$ algebraically independent meromorphic functions $\hat f_1,\ldots, \hat f_{n-2}$ on $M_A$, induced by the $F_A$-invariant meromorphic functions on $\CC^n$:
\begin{equation*}
f_j(z,w)=\frac{z_j}{w_{n-1}w_n}, \ \ (z,w)\in\CC^{n-2}\times\CC^2, \  \ 1\leq j\leq n-2.
\end{equation*} 
Here, $z=(z_1,\ldots, z_{n-2})$ denote the first $n-2$ holomorphic coordinates on $\CC^n$ and $w=(w_{n-1},w_n)$ denote the last two holomorphic coordinates. Since $a(M_A)=n-2$, it follows that the field of meromorphic functions on $M_A$, $\CC(M_A)$, is a finite algebraic extension of $\CC(\hat f_1,\ldots, \hat f_{n-2})$.

\begin{thm}\label{ThmAlgRed}
Let $\Phi$ be the meromorphic map: 
\begin{equation}\label{algred2}
\Phi:M_A\dashrightarrow \CC\PP^{n-2}, x\mapsto [1:\hat f_1(x):\ldots: \hat f_{n-2}(x)].
\end{equation}
Then there exists a modification $\mu:M_A^*\rightarrow M_A$ such that the map $r=\Phi\circ\mu:M_A^*\rightarrow \CC\PP^{n-2}$ is the algebraic reduction of $M_A$. Its generic fiber is bimeromorphic to $M_B$. In particular, $\CC(M_A)=\CC(\hat f_1,\ldots, \hat f_{n-2})$. 
\end{thm}
\begin{proof}
Let us start by noting that the map $\Phi$ becomes holomorphic when restricted to $M_A-C_A$. Proposition~\ref{compactification} and Lemma~\ref{stablerelation} give us an identification $M_A-C_A\cong (\CC^{n-2}\times W^s(F_B)^*)/U_A$, with $W^s(F_B)^*\subset \CC^2$, with respect to which the map  $\Phi$ becomes:  
\begin{equation}\label{algred1}
\hat\Phi:M_A-C_A\rightarrow \CC\PP^{n-2}, \widehat{(z,w)}\mapsto [w_{n-1}w_n:z_1:\ldots:z_{n-2}], \ \ (z,w)\in\CC^{n-2}\times W^s(F_B)^*
\end{equation}
where we have denoted by $\widehat{(z,w)}$ the class of $(z,w)$ modulo the action of $U_A$. Moreover, under the isomorphism $M_A-C_A\cong \CC^{n-2}\times (M_B-C_B)$ given by Theorem~\ref{Thmcompact}, this map is precisely the composition of the natural projection  $\CC^{n-2}\times (M_B-C_B)\rightarrow\CC^{n-2}$ with the map $\CC^{n-2}\rightarrow\CC\PP^{n-2}$, $z\mapsto [1:z]$. 

In particular, it follows that $\overbar{\Phi(M_A-C)}=\CC\PP^{n-2}$, i.e. $\Phi$ is a dominant map, and it induces an injection between the corresponding fields of meromorphic functions $\Phi^*:\CC(T_1,\ldots, T_{n-2})\rightarrow \CC(M_A)$. Moreover, for a point $t=[1:t_1:\cdots:t_{n-2}]\in\CC\PP^{n-2}$, the fiber $\hat\Phi^{-1}(t)\cong M_B-C_B$ is connected. 

Let now $R:\hat M_A\rightarrow V$ be the algebraic reduction of $M_A$, let $\hat\mu:\hat M_A\rightarrow M_A$ denote the corresponding modification and let $\hat R:M_A\dashrightarrow V$ be a dominant meromorphic map with $R=\hat R\circ\hat\mu$. It follows that $u:=(R^*)^{-1}\circ\Phi^*:\CC(\CC\PP^{n-2})\rightarrow \CC(V)$ is an inclusion of fields, hence, as both $\CC\PP^{n-2}$ and $V$ are projective manifolds, $u$ corresponds to a dominant rational map $\Phi_u:V\dashrightarrow\CC\PP^{n-2}$ with $\Phi^*_u=u$. As $\dim V=\dim \CC\PP^{n-2}$, a generic fiber of $\Phi_u$ consists in a finite number of points. 

On the other hand,  we must have $\Phi_u\circ \hat R=\Phi$. Indeed, this follows from the fact that $(\Phi_u\circ \hat R)^*T_j=\hat f_j$ and $T_j=\frac{t_j}{t_0}$ for any $1\leq j\leq n-2$, where $t_0,\ldots,t_{n-2}$ denote the homogeneous coordinates on $\CC\PP^{n-2}$. Moreover, the generic fibers of $\hat R$ and of $\Phi$ are connected. We infer thus that also the generic fiber of $\Phi_u$ is connected, hence $\Phi_u$ is a bimeromorphism.  

We conclude that indeed $\Phi:M_A\dashrightarrow\CC\PP^{n-2}$ is a model for the algebraic reduction. Namely, by the elimination of indeterminacies of $\Phi$  \cite{hirn2}, there exists a modification $\mu:M_A^*\rightarrow M_A$ and a holomorphic map $r:M_A^*\rightarrow \CC\PP^{n-2}$ so that $r=\Phi\circ\mu$. Thus $r$ is the desired algebraic reduction.
\end{proof}

\begin{cor}
Let $A\in\GL(n,\ZZ)$ be a Kato matrix of type $l=n-2$ with $\det A=1$. Then:
\begin{equation*}
 \dim H^0(M_A,K_{M_A}^*)={2n-3 \choose n-2}=\dim H^0(\CC\PP^{n-2},\mathcal O(n-1)).
 \end{equation*}
\end{cor}
\begin{proof}
By Proposition~\ref{Kodairadimension}, the line bundle $K_{M_A}^*$ is associated to the divisor $D=(n-1)C_A$. Hence 
\begin{equation*}
H^0(M,K_{M_A}^*)\cong \{f\in\CC(M_A), div(f)+D\geq 0\}=:\mathcal L(D).
\end{equation*}
On the other hand, by Theorem~\ref{ThmAlgRed} we have: 
\begin{equation*}
\CC(M_A)=\CC(\hat f_1,\ldots, \hat f_{n-2})\cong\{\frac{P}{Q},\ P,Q\in\CC[z_0,\ldots, z_{n-2}] \text{ homogeneous, } \mathrm{deg} P=\mathrm{deg} Q\}
\end{equation*}
where we put $z_0=w_{n-1}w_n$ and $\deg z_j=1$, $0\leq j\leq n-2$. 

Since $C_A$ is induced by the divisor $w_{n-1}w_n=0$ on $\BB$, it follows that a meromorphic function represented as $f=\frac{P}{Q}$ is in $\mathcal L(D)$ if and only if $Q|z_0^{n-1}$. We thus find:
\begin{align*}
\mathcal{L}(D)\cong H^0(\CC\PP^{n-2},\mathcal O(n-1))
\end{align*} 
which is well-known to be of dimension ${2n-3\choose n-2}$.
\end{proof}

In the case $l=n-2$, we also have a full description of the space of holomorphic vector fields of the manifolds $M_A$:
 
\begin{thm}\label{ThmHvn-2}
Let $A\in\GL(n,\ZZ)$ be a Kato matrix of type $n=l-2$, with $n\geq 3$, and let $M_A$ be the associated Kato manifold. Then the space $H^0(M_A,TM_A)$ is of complex dimension $(n-2)(n-1)$, generated by the $F_A$-invariant vector fields on $\CC^n$:
\begin{align*}
X_{s,t}=&z_s\frac{\del}{\del z_t}, \ \ 1\leq s,t\leq n-2\\
Y_{j}=&w_{n-1}w_n\frac{\del}{\del z_j}, \ \ 1\leq j\leq n-2.
\end{align*}
\end{thm}
\begin{proof}
The proof is a continuation of the proof of Proposition~\ref{vf}. First of all, using that $J_0B-J_0=L$, it easy to check that the vector fields defined above are $F_A$ invariant and vanish at $0$, hence they define linearly independent holomorphic vector fields on $M_A$.

Conversely, we wish to show that $\dim H^0(M_A,TM_A)\leq (n-2)(n-1)$. By Lemma~\ref{cov}, for any $p\geq 1$, $M_{A^p}$ is a finite covering of $M_A$, hence $\dim H^0(M_{A^p},TM_{A^p})\geq \dim H^0(M_A,TM_A)$. Thus it suffices to prove the desired inequality for a power of $A$, so we can suppose that $A$ is $l$-positive and $B$ is positive.

Using the considerations made during the proof of Proposition~\ref{vf}, we know that any holomorphic vector field on $M_A$ defines an $F_A$-invariant vector field on $\BB\subset\CC^n$ of the form:
\begin{equation*}
X_{(z,w)}=\sum_{j=1}^{n-2}h_j(z,w)\frac{\del}{\del z_j}+h_{n-1}(z,w)w_{n-1}\frac{\del}{\del w_{n-1}}+h_{n}(z,w)w_{n}\frac{\del}{\del w_{n}}, \ \ (z,w)\in\DD\subset\CC^{n-2}\times\CC^2
\end{equation*}
where $h_1,\ldots, h_{n}$ are holomorphic functions on $\BB$. The $F_A$ invariance of $X$ translates into the following conditions on these functions:
\begin{align}
F_A^*h_j&=w^L(h_j+l_{n-1}h_{n-1}+l_nh_n), \ 1\leq j\leq n-2 \label{eqVf1}\\
F_A^*h_j&=b_{jn-1}h_{n-1}+b_{jn}h_n, \ n-1\leq j\leq n \label{eqVf2}
\end{align}
where $B=(b_{st})_{n-1\leq s,t\leq n}$ and $L=(l_{n-1},l_n)$.
For $1\leq j\leq n$, let us write the function $h_j$ as a power series 
\begin{equation*}
h_j(z,w)=\sum_{\substack{I\in\NN^{n-2}\\ J\in\NN^2}}d_{I,J}^{(j)}z^Iw^J
\end{equation*} 
which converges absolutely and uniformly on $\BB$.

Then eq.~\eqref{eqVf2} reads:
\begin{equation*}
\sum_{J\in\NN^2}d_{I,J}^{(j)}w^{|I|L+JB}=\sum_{K\in\NN^2}(b_{jn-1}d^{(n-1)}_{I,K}+b_{jn}d^{(n)}_{I,K})w^K, \ \ n-1\leq j\leq n, \ I\in\NN^{n-2}
\end{equation*}
where for $I=(I_1,\ldots,I_{n-2})\in\NN^{n-2}$ we put $|I|=I_1+\ldots+I_{n-2}$. For $I=0$ we find as before that for any $J\neq 0\in\NN^2$, $d_{0,J}^{(n-1)}=d_{0,J}^{(n)}=0$, while for $J=0$, we find that the vector $d=(d^{(n-1)}_{0,0},d^{(n)}_{0,0})^t$ satisfies $Bd=d$, thus $d=(0,0)^t$. For $I\neq 0$ we also find the vanishing of the coefficients. Indeed, if this is not the case then we can take $J\in\NN^2$ minimal (for the lexicographic order) so that there exists $j\in\{n-1,n\}$ with $d_{I,J}^{(j)}\neq 0$. But then we find: 
\begin{equation*}
d^{(j)}_{I,J}=b_{jn-1}d^{(n-1)}_{I,K}+b_{jn}d^{(n)}_{I,K}, \ K=(J-|I|L)B^{-1}.
\end{equation*}
However clearly $K<|I|L+KB=J$ and $d_{I,K}^{(s)}\neq 0$ for some $s\in\{n-1,n\}$ as $B$ is invertible, which contradicts the minimality of $J$. Thus $h_{n-1}=h_{n}=0$.

Now eq.~\eqref{eqVf1} becomes:
\begin{equation*}
\sum_{J\in\NN^2}d^{(j)}_{I,J}w^{|I|L+JB}=\sum_{K\in\NN^2}d^{(j)}_{I,K}w^{L+K}, \ 1\leq j\leq n-2, \ I\in\NN^{n-2}.
\end{equation*}
Suppose that for some $j\in\{1,\ldots, n-2\}$, $I\in\NN^{n-2}$ and $J\in\NN^2$ we have $d^{(j)}_{I,J}\neq 0$. Let us suppose by contradiction that $(I,J)\neq (0,J_0)$ and that if $J=0$, then $|I|\neq 1$. Then the sequence defined by
\begin{align*}
J_{-m}&=JB^{-m}+(1-|I|)L(B^{-m}+B^{-m+1}+\ldots+B^{-1}), \\
J_m&=JB^m+(|I|-1)L(B^{m-1}+\ldots+I_2), \ \ m\in\NN
\end{align*}
must satisfy $J_m\in\NN^2$ and $d_{I,J}=d_{I,J_m}$ for any $m\in\ZZ$. Note that all the terms of the sequence $(J_m)_m$ are distinct. Indeed, if there existed $m>p$ with $J_m=J_p$, then one immediately finds that $J(B-I_2)=(1-|I|)L$, which then implies that $J=(1-|I|)J_0$ as $B-I_2$ is invertible. It follows that either $J=0$ and $|I|=1$ or $(I,J)=(0,J_0)$. However we have excluded both these cases.

Hence both sequences $(J_m)_{m\geq 0}$ and $(J_m)_{m\leq 0}$ are infinite and it follows from the normal convergence of the series defined by $h_j$ that the series $\sum_{m\geq 0}w^{J_m}$ and $\sum_{m\leq 0}w^{J_m}$ converge absolutely for any $w\in\DD\subset \CC^2$. Now, like in the proof of Proposition~\ref{holofunct}, for $w=p(u)=(\e^{2\pi u_{n-1}},\e^{2\pi u_n})\in\DD$ with $u\in\CC^2$, we find:
\begin{equation*}
w^{J_{-m}}=\e^{2\pi(\langle J^*,B^{-m}u\rangle+(1-|I|)\langle L^*, (B^{-m}+\ldots+B^{-1})u\rangle)}, \ \ m\in\NN
\end{equation*}
and a similar formula for $w^{J_m}$. Hence the functions: 
\begin{align*}
v_m(u)&=\re\langle J^*,B^{-m}u\rangle+(1-|I|)\re\langle L^*, (B^{-m}+\ldots+B^{-1})u\rangle\\
u_m(u)&=\re\langle J^*,B^{m}u\rangle+(|I|-1)\re\langle L^*, (B^{m-1}+\ldots+I_2)u\rangle)
\end{align*}
must tend to $-\infty$ when $m$ goes to $\infty$ for any $u\in\CC^2$ with $p(u)\in\DD$. But for a Perron vector $f$ with positive components for which $p(f)\in\DD$ and $Bf=\al f$, $\al>1$, we find:
\begin{align*}
v_m(f)&=\al^{-m}\langle J^*,f\rangle+(1-|I|)(\al^{-m}+\ldots+\al^{-1})\langle L^*,f\rangle\\
u_m(f)&=\al^{m}\langle J^*,f\rangle+(|I|-1)(\al^{m-1}+\ldots+1)\langle L^*,f\rangle.
\end{align*}
Now, since $\al>1$ and $f$, $J$ and $L$ have non-negative components, we find that for $I=0$, $v_m(f)\geq 0$ for any $m>0$, while for $I\neq 0$, $u_m(f)\geq 0$ for any $m>0$. Thus we find a contradiction, implying that $d^{(j)}_{I,J}=0$. Finally, we conclude that the vector field $X$ must be a linear combination of the vector fields $X_{s,t}$ and $Y_j$, $1\leq j,s,t\leq n-2$.
\end{proof}

\begin{rmk}
Note that the vector fields in the above result define a singular holomorphic foliation $\mathcal{F}$ on $M_A$ of rank $n-2$ with the property that a leaf of $\mathcal F$ passing through a point of $C_A$ is contained in $C_A$. On the other hand, $\mathcal F$ is regular outside $C_A$, with leaves biholomorphic to $\CC^{n-2}$.  Moreover, by Corollary~\ref{torus}, we have an effective holomorphic action of $\mathbb{T}^{n-2}$ on $M_A$.  
\end{rmk}

\subsection*{Acknowledgments} We are very grateful to Georges Dloussky, Matteo Ruggiero and Victor Vuletescu for many stimulating discussions and useful suggestions.

\end{document}